\documentclass[10pt]{article}

\usepackage{mathptmx}

\usepackage{mathtools}

\usepackage[utf8]{inputenc}
\usepackage{array}
\usepackage{booktabs}
\usepackage{amsmath}
\usepackage{amsfonts}
\usepackage{amssymb}
\usepackage{amsthm}
\usepackage{extarrows}
\usepackage{hyperref}
\usepackage{mathrsfs}
\usepackage{tikz}
\usepackage{tikz-cd}
\usepackage{mathtools}
\usepackage{graphicx,multirow}
\usepackage{dsfont}
\usepackage{authblk}
\usepackage{float}

\usepackage[bottom]{footmisc}

\usepackage{subcaption}

\theoremstyle{plain}
\newtheorem{theorem}{Theorem}[section]
\newtheorem*{theorem*}{Theorem}
\newtheorem{theorem-non}{Theorem}
\newtheorem{proposition}{Proposition}[section]
\newtheorem{lemma}{Lemma}[section]
\newtheorem{corollary}{Corollary}[section]

\newtheorem{definition}{Definition}[section]

\theoremstyle{definition}

\numberwithin{equation}{section}

\usepackage[pagewise]{lineno}

\usepackage{xcolor}

\theoremstyle{remark}
\newtheorem{remark}[theorem]{Remark}
\newtheorem{example}{Example}[section]

\normalsize
\numberwithin{equation}{section}

\begin{document}

\title{Lax formalism for Gelfand-Tsetlin integrable systems}

%\author{Eder M .Correa}
%\address{IMPA, Estrada Dona Castorina 110, Rio de Janeiro, 22460-320, Brazil}
%\email{edermoraes@impa.br}
%
%\author{Lino Grama}
%\address{IMECC-Unicamp, Departamento de Matem\'{a}tica. Rua S\'{e}rgio Buarque de Holanda,
%651, Cidade Universit\'{a}ria Zeferino Vaz. 13083-859, Campinas - SP, Brazil}
%%\address{E-mail: {\rm linograma@gmail.com.}}
%\email{linograma@gmail.com}

\author[1]{Eder M. Correa\thanks{Eder M. Correa was supported by CNPq grant 150899/2017-3.}}
\author[2]{Lino Grama\thanks{Lino Grama is partially supported by FAPESP grant 2018/13481-0, and by CNPq grant 2019/305036-0.}}
\affil[1]{UFMG, Avenida Ant\^{o}nio Carlos, 6627, 31270-901 Belo Horizonte - MG, Brazil.}
\affil[2]{IMECC-Unicamp, Departamento de Matem\'{a}tica. Rua S\'{e}rgio Buarque de Holanda,
651, Cidade Universit\'{a}ria Zeferino Vaz. 13083-859, Campinas - SP, Brazil.}

\date{}

\setcounter{Maxaffil}{0}
\renewcommand\Affilfont{\itshape\small}

%\linenumbers

\maketitle
\begin{abstract} 
In the present work, we study Hamiltonian systems on (co)adjoint orbits and propose a Lax pair formalism for Gelfand-Tsetlin integrable systems defined on (co)adjoint orbits of the compact Lie groups ${\rm{U}}(n)$ and ${\rm{SO}}(n)$. In the particular setting of (co)adjoint orbits of ${\rm{U}}(n)$, by means of the associated Lax matrix we construct a family of algebraic curves which encodes the Gelfand-Tsetlin integrable systems as branch points. This family of algebraic curves enables us to explore some new insights into the relationship between the topology of singular Gelfand-Tsetlin fibers, singular algebraic curves and vanishing cycles. Further, we provide a new description for Guillemin and Sternberg’s action coordinates in terms of hyperelliptic integrals.
\end{abstract}

\tableofcontents

\section{Introduction}

The notion of Lax pair is a new emergent language used in the study of integrability, and one of the most important features of this concept is its relation with the classical $r$-matrix, see \cite{QUANTUM}. The classical $r$-matrix was introduced in late 1970’s by Sklyanin \cite{SKY}, as a part of a vast research program launched by L. D. Faddeev, which culminated in the discovery of the Quantum Inverse Scattering Method and of Quantum Groups, e.g. \cite{FEDDEV}. A Lax pair for a Hamiltonian system $(M,\omega,H)$, consists of two matrix-valued functions $(L,P)$ on the phase space $(M,\omega)$ of the system, such that the Hamiltonian evolution equation of motion associated to a Hamiltonian $H \in C^{\infty}(M)$ can be written as a ``zero curvature equation", also known as {\em{Lax equation}}, see for instance \cite{Lax} and \cite{Zakharov}. For more details about the relationship between the notion of zero curvature equation and the Lax equation, see for instance \cite[Chapter 9]{Guest}. In a broad sense, in which one allows the Lax pairs depending on a complex parameter $\zeta$, called the spectral parameter, i.e., $L = L(\zeta)$ and $P = P(\zeta)$, through the notion of {\textit{spectral curves}}\footnote{As mentioned in \cite{Donagi}, the notion of spectral curves arose historically out of the study of differential equations of Lax type. Following Hitchin’s work \cite{Hitchin}, they have acquired a central role in understanding the moduli spaces of vector bundles and Higgs bundles on a curve. } we have a very rich interplay between integrable Hamiltonian systems and algebraic curves. Interesting examples in this setting include periodic Toda lattices based on simple Lie algebras, the Neumann problem of the motion of a point on the sphere under a linear force, the free motion of a point on an ellipsoid, the Euler, Lagrange, and Kovalevskaya tops, and other integrable systems, see for instance \cite{Griffiths}, \cite{Adler}, \cite{Moser}, \cite{Beauville}, \cite{INTCLASSICALSYSTEM}, and references therein.

In this paper, we deal with the problem related to the formulation of the {\em Gelfand-Tsetlin integrable systems} on coadjoint orbits in terms of Lax pairs. The Gelfand-Tsetlin system was introduced in the 1980’s by Guillemin-Sternberg \cite{GELFORB}. They showed that for coadjoint orbits of the compact unitary Lie group ${\rm{U}}(n)$ the set of Poisson-commuting functions provided by a suitable application of Thimm's trick \cite{THIMM} defines a completely integrable system. The key point for the integrability of the Gelfand-Tsetlin systems is that the coadjoint action of ${\rm{U}}(n-1)$ on a coadjoint orbit of ${\rm{U}}(n)$ is {\em{coisotropic}}, see for instance \cite{INTMULT}. This last property also holds for coadjoint orbits of ${\rm{SO}}(n)$ (see also \cite{Heckman}) and, recently, a modern proof for the integrability of Gelfand-Tsetlin systems on coadjoint orbits of ${\rm{U}}(n)$ and ${\rm{SO}}(n)$ was given in \cite{Panyushev}. A remarkable feature of the Gelfand-Tsetlin systems is their relation with representation theory of Lie groups, Lie algebras, and geometric quantization, see for instance \cite{QUANTMULTFREE}. Also, this class of integrable systems has been playing an important role as a concrete model for the study of Lagrangian Floer theory on non-toric manifolds, e.g. \cite{NNU}, \cite{NU}, \cite{Cho}. 

As mentioned in \cite{Izosimov}, the algebraic geometric framework based on the notion of a Lax representation has proved to be very powerful not only for constructing new examples and explicit integration, but also for studying topological properties of integrable systems (e.g., \cite{Audin}, \cite{Audin1}). Motivated by these ideas and following the results provided in \cite{Cho} and \cite{Bouloc}, we also explore some new insights in the study of the relationship between the topology of singular Gelfand-Tsetlin fibers, singular algebraic curves and vanishing cycles. 

\subsection{Main results}

Our first result is concerned with the Lax pair formalism for the Gelfand-Tsetlin systems. The main contribution of this result is to provide a canonical and concrete way to assign Lax pairs to Gelfand-Tsetlin integrable systems for (co)adjoint orbits of ${\rm{U}}(n)$ and ${\rm{SO}}(n)$. In order to state our first result, let us introduce some basic concepts. Let $G$ be a compact connected Lie group with Lie algebra $\mathfrak{g}$. For every adjoint orbit $O(\Lambda) \subset \mathfrak{g}$, let $(O(\Lambda),\omega_{O(\Lambda)},G,\Phi)$ be the (canonical) associated Hamiltonian $G$-space, where $\omega_{O(\Lambda)}$ is the {\em Kirillov-Kostant-Souriau symplectic form}, and $\Phi \colon (O(\Lambda),\omega_{O(\Lambda)}) \to \mathfrak{g}$ is the moment map defined by the natural Hamiltonian action of $G$ on $(O(\Lambda),\omega_{O(\Lambda)})$. In this setting, we prove the following result:
\begin{theorem-non}
\label{Maintheorem}
Given a Hamiltonian $G$-space $(O(\Lambda),\omega_{O(\Lambda)},G,\Phi)$, where $G$ is either ${\rm{U}}(n) \  \text{or} \ \ {\rm{SO}}(n)$, then there exist $H \in C^{\infty}(O(\Lambda))$, and a pair of matrix-valued functions $L,P \colon (O(\Lambda),\omega_{O(\Lambda)}) \to \mathfrak{gl}(r,\mathbb{R})$, satisfying the Lax equation 
\begin{equation}
 \frac{d}{dt}L(\varphi_{t}(Z)) + \big [L(\varphi_{t}(Z)), P(\varphi_{t}(Z)) \big ] = 0,
\end{equation}
for all $Z \in O(\Lambda)$, where $\varphi_{t}(Z)$ denotes the flow of the Hamiltonian vector field $X_{H}$ through $Z \in O(\Lambda)$. Moreover, the solutions of the spectral equation
\begin{equation}
\label{speceq}
\det\big (w \mathds{1}_{r} - L\big) = 0, 
\end{equation}
provide a maximal set of conserved quantities in involution for the Hamiltonian system $(O(\Lambda),\omega_{O(\Lambda)},H)$ which coincides with the Gelfand-Tsetlin integrable system on $O(\Lambda)$.
\end{theorem-non}

The result above establish a description for Gelfand-Tsetlin integrable systems using only techniques involving Lax matrices. The ideas developed in our construction are quite natural and generalize in a suitable sense some results introduced in \cite{Giacobbe} for regular (co)adjoint orbits of ${\rm{U}}(n)$, see Remark \ref{RemarkGiacobbe}. As pointed out in \cite{Cho}, the Gelfand-Tsetlin system resembles a toric moment map in the sense that its image is a convex polytope $\Delta_{\Lambda}$ and the fiber over every interior point of $\Delta_{\Lambda}$ is a Lagrangian torus. The notable difference is that non-torus Lagrangian fibers may appear at some boundary strata of $\Delta_{\Lambda}$. In the general setting, as observed in \cite{BolsinovOshemkov}, the classical Arnold-Liouville theorem completely describes the Liouville foliation defined by an integrable system in a neighborhood of its regular leaves, but says almost nothing about its structure near singular leaves. In some sense, all topological properties of the system are determined by the structure of its singularities. The local properties of singularities of integrable systems were studied by several authors through normal forms, e.g. \cite{Vey}, \cite{Ito}, \cite{Eliasson}, \cite{MirandaZung}. In particular, several classical examples, such as Kovalevskaya top \cite{JPF}, Euler top \cite{JPFTarama}, the pendulum \cite{TaramaJPF} and the Clebsch top \cite{JPFTarama1}, provide a very rich interplay between certain elliptic fibrations and the Birkhoff normal forms, see also \cite{NarukiTarama}. In this sense, as an application of the ideas introduced in Theorem \ref{Maintheorem}, our second result aims to provide some new tools to investigate the behavior of the Liouville foliation defined by Gelfand-Tsetlin systems through certain families of complex algebraic curves. Based on the classification of Gelfand-Tsetlin fibers provided in \cite{Cho} by means of the combinatorics of ladder diagrams, we prove the following result:
\begin{theorem-non}
\label{Theo2}
Let $(O(\Lambda),\omega_{O(\Lambda)},\mathscr{H})$ be the Gelfand-Tsetlin integrable system associated to some adjoint orbit $O(\Lambda) \subset \mathfrak{u}(n)$ and let $\Delta_{\Lambda}$ be the corresponding Gelfand-Tsetlin polytope. Then, there exists a family of complex algebraic curves 
\begin{equation}
\textstyle{\mathcal{M}_{\Lambda} := \Big \{\textstyle{{\mathcal{C}_{Z} = {\rm{Spec}}}\Big(\frac{\mathbb{C}[w,\zeta]}{\langle \zeta^{2} - P_{\Lambda}(w;Z) \rangle}\Big ) }\ \Big | \ Z \in O(\Lambda) \Big\}},
\end{equation}
such that $P_{\Lambda}(w;Z) := r_{\Lambda}(w)\det( w\mathds{1}_{r} - L(Z))$, for every $Z \in O(\Lambda)$, and $r_{\Lambda}(w) \in \mathbb{C}(w)$, satisfying the following:
\begin{enumerate}
\item[1)] For all $Z \in O(\Lambda)$, we have $P_{\Lambda}(w;Z) = \mu_{\Lambda}(w)P_{1}(w;Z) \cdots P_{n-1}(w;Z)$, such that $\mu_{\Lambda}(w)$ is the minimal polynomial of $\Lambda$ and $P_{k}(w;Z) \in \mathbb{C}[w]$, $\forall k = 1,\ldots, n-1$;
\item[2)] Given ${\bf{u}} \in \Delta_{\Lambda}$, then $\mathscr{H}^{-1}({\bf{u}})$ is a Lagrangian torus if and only if $\mathcal{C}_{Z}$ is a smooth hyperelliptic curve, for some $Z \in \mathscr{H}^{-1}({\bf{u}})$;
\item[3)] If ${\bf{u}} \in {\rm{int}}(\mathcal{F})$, for some $d$-dimensional face $\mathcal{F}$ of $\Delta_{\Lambda}$, then $\mathcal{C}_{Z}$ is singular for all $Z \in \mathscr{H}^{-1}({\bf{u}})$;
\item[4)] In the setting of item 3), if $0 \leq {\rm{Res}}\Big(\frac{P'_{k}(w;Z)}{P_{k}(w;Z)};\alpha\Big) \leq 1$, for all $(\alpha,0) \in {\rm{Sing}}(\mathcal{C}_{Z})$, and for all $1 \leq k \leq n-1$, then $\mathscr{H}^{-1}({\bf{u}}) \cong T^{d}$; 
\item[5)] If ${\bf{u}} \in {\rm{int}}(\mathcal{F})$, for some $d$-dimensional face $\mathcal{F}$ of $\Delta_{\Lambda}$, and $\mathscr{H}^{-1}({\bf{u}})$ is a non-torus fiber, then there exists some $(\alpha,0) \in {\rm{Sing}}(\mathcal{C}_{Z})$, for some $Z \in \mathscr{H}^{-1}({\bf{u}})$, such that ${\rm{mult}}_{\alpha}(P_{k}(w;Z)) > 1$, for some $1 \leq k \leq n-1$;
\item[6)] For all $Z \in O(\Lambda)^{\mathscr{H}}$ there exists an Abelian differential $\Omega(Z) \in \mathscr{M}^{1}(\mathcal{C}_{Z})$, with poles $p_{jk}(Z) \in \mathcal{C}_{Z}$, $(j,k) \in \mathcal{I}(\Lambda)$, such that 
\begin{equation}
\mathscr{H}(Z) = \Bigg (\frac{1}{2\pi\sqrt{-1}}\oint_{c_{jk}}\Omega(Z)\Bigg)_{(j,k) \in \mathcal{I}(\Lambda)},
\end{equation}
where $c_{jk}$ is a small loop around the point $p_{jk}(Z) \in \mathcal{C}_{Z}$, for every $(j,k) \in \mathcal{I}(\Lambda)$.
\end{enumerate}
\end{theorem-non}
The family of algebraic curves $\mathcal{M}_{\Lambda}$ introduced in the above theorem enables us to explore by means of the concrete low dimensional standard examples $\mathbb{C}{\rm{P}}^{1}$, ${\rm{Fl}}(3)$ and ${\rm{Gr}}(2,4)$, some new insights into the relationship between the topology of singular Gelfand-Tsetlin fibers, singular algebraic curves and vanishing cycles. More precisely, we classify Gelfand-Tsetlin fibers for the aforementioned examples using elementary tools coming from the theory of algebraic curves and we provide a characterization for singular Gelfand-Tsetlin fibers in terms of the vanish of certain homology cycles of a generic smooth element in $\mathscr{M}_{\Lambda}$. Our analysis suggests that questions concerned with the topology and dynamics of the Gelfand-Tsetlin systems, specially those related to the singular locus of the underlying Liouville foliation, can be approached using the theory of algebraic curves and related topics. 

%\begin{figure}[H]
%\centering
%\includegraphics[scale=.25]{GT_P_curves1.jpg}
%\caption{Algebraic curves parameterized by points of the Gelfand-Tsetlin polytope $\Delta_{\Lambda} \subset \mathbb{R}^{3}$ associated to a regular adjoint orbit  $O(\Lambda) \subset \mathfrak{u}(3)$, such that $\Lambda = {\text{diag}}(2\sqrt{-1},0,-2\sqrt{-1})$. The point $A$ represents an interior point, whereas $B$,$C$, and $D$, represent three vertices of $\Delta_{\Lambda} \subset \mathbb{R}^{3}$.}
%\label{polytopecurves}
%\end{figure}

%\begin{theorem-non}
%Let $(O(\Lambda),\omega_{O(\Lambda)},G,\Phi)$ be a Hamiltonian $G$-space defined by an adjoint orbit $O(\Lambda) = {\Ad}(G)\Lambda \subset \mathfrak{g}$, where $G = {\rm{U}}(n)$ or ${\rm{SO}}(n)$. Then there exists a Lax pair $L,P \colon O(\Lambda) \to \mathfrak{gl}(r,\mathbb{R})$ satisfying 
%\begin{center}
%
%$\displaystyle \frac{d}{dt}L + \big [L,P \big] = 0$,
%
%\end{center}
%such that the spectral invariants of $L$ define an integrable system in $O(\Lambda)$. Furthermore, this integrable system coincides with the Gelfand-Tsetlin integrable system.
%
%\end{theorem-non}

%%%%%%%%%%%%%%%%%%%%%%%%%%%%%%%%%%%%%%%%%%%%%%%%%%%%%%%%%%%%%%%%%%%%%%%%%%%%%%%%%%%%%%%%%%%%%%%%%%%%%%%%%%%%%%%%%%%%%%%%%%%%%%%%%%%%%%%%%%%%%%%%%%%%%%%%%%%%%%%%%%%%%%%%%

\section{Collective Hamiltonians and Gelfand-Tsetlin integrable systems}

In this section, we provide a brief overview about Hamiltonian systems defined by collective Hamiltonians. The main purpose is to cover the basic material in this topic and explain Thimm's trick \cite{THIMM}.

%%%%%%%%%%%%%%%%%%%%%%%%%%%%%%%%%%%%%%%%%%%%%%%%%%%%%%%%%%%%%%%%%%%%%%%%%%%%%%%%%%%%%%%%%%%%%%%%%%%%%%%%%%%%%%%%%%%%%%%%%%%%%%%%%%%%%%%%%%%%%%%%%%%%%%%%%%%%%%%%%%%%%%%%%

\subsection{Collective Hamiltonian systems}

%%%%%%%%%%%%%%%%%%%%%%%%%%%%%%%%%%%%%%%%%%%%%%%%%%%%%%%%%%%%%%%%%%%%%%%%%%%%%%%%%%%%%%%%%%%%%%%%%%%%%%%%%%%%%%%%%%%%%%%%%%%%%%%%%%%%%%%%%%%%%%%%%%%%%%%%%%%%%%%%%%%%%%%%%

Let $(M,\omega)$ be a symplectic manifold and $\tau \colon G \to {\text{Diff}}(M)$ be a smooth action. The action $\tau$ is said to be Hamiltonian  if and only if it admits a moment map $\Phi \colon (M,\omega) \to \mathfrak{g}^{\ast}$, see for instance \cite[\S 10.1]{Rudolph}. We say that the moment map $\Phi$ associated to $\tau$ is \textit{equivariant} if it satisfies
$$\Phi(\tau(g)p) = {\text{Ad}}^{\ast}(g)\Phi(p),$$
for every $p \in M$, $X \in \mathfrak{g}$ and $g \in G$. In this work we are concerned with the following setting.

%In what follows we fix a basic data which is the content of the following definition:

\begin{definition} A Hamiltonian $G$-space $(M,\omega,G,\Phi)$ is composed of: 
\begin{itemize}
\item A symplectic manifold  $(M,\omega)$ and a connected Lie group $G$, with Lie algebra $\mathfrak{g}$.

\item A Hamiltonian (left) Lie group action $\tau \colon G \to {\text{Diff(M)}}$, with associated infinitesimal action $\delta \tau\colon \mathfrak{g} \to \mathfrak{X}(M)$.

\item A moment map $\Phi \colon (M,\omega) \to \mathfrak{g}^{\ast}$.

\end{itemize}

\end{definition}

\begin{definition}

A Hamiltonian system is defined by a triple $(M,\omega,H)$, where $(M,\omega)$ is a symplectic manifold and $H \in C^{\infty}(M)$.

\end{definition}

Given a compact connected Lie group $G$, with Lie algebra $\mathfrak{g}$, we are interested in the study of a certain class of Hamiltonian systems associated to Hamiltonian $G$-spaces $(O(\lambda),\omega_{O(\lambda)},G,\Phi)$, such that\footnote{For more details see \cite[\S 6.1.1-6.1.2]{Singspaces}} $O(\lambda) \subset \mathfrak{g}^{\ast}$ is the coadjoint orbit through some $\lambda \in \mathfrak{g}^{\ast}$, $\omega_{O(\lambda)}$ is the Kirillov-Kostant-Souriau symplectic form and $\Phi \colon O(\lambda) \to \mathfrak{g}^{\ast}$ is defined by the natural inclusion map $O(\lambda) \hookrightarrow \mathfrak{g}^{\ast}$.
%\begin{enumerate}

%\item $O(\lambda) = \Big \{ {\text{Ad}}^{\ast}(g)\lambda \in \mathfrak{g}^{\ast} \ \Big | \ g \in G \Big \} $;

%\item $\omega_{O(\lambda)}({\text{ad}}^{\ast}(X)\xi,{\text{ad}}^{\ast}(Y)\xi) = - \xi ([X,Y])$, $\forall \xi \in O(\lambda)$ e $\forall X,Y \in \mathfrak{g}$;

%\item $\tau \colon G \to {\text{Diff}}(O(\lambda))$, such that $\tau(g)\phi = {\text{Ad}}^{\ast}(g)\phi$;

%\item $\delta \tau \colon \mathfrak{g} \to \Gamma(TO(\lambda)))$, where $\delta \tau(X)\phi = {\text{ad}}^{\ast}(X)\phi$;

%\item The momentum map $\Phi \colon O(\lambda) \to \mathfrak{g}^{\ast}$ associated to the Hamiltonian actions $\tau$ is defined by the natural inclusion $O(\lambda) \hookrightarrow \mathfrak{g}^{\ast}$, so it is an equivariant moment map. 

%\end{enumerate} 
\begin{remark}
\label{Adjointorbit}
By fixing an ${\rm{Ad}}$-invariant inner product $(-,-)$ on $\mathfrak{g}$, we obtain an isomorphism $\mathfrak{g}\simeq \mathfrak{g}^\ast$. From this, given $(O(\lambda),\omega_{O(\lambda)},G,\Phi)$, we can always consider $(O(\Lambda),\omega_{O(\Lambda)})$, where $\lambda = (\Lambda,-)$, for some $\Lambda \in \mathfrak{g}$, such that $O(\Lambda) \subset \mathfrak{g}$ is the adjoint orbit through $\Lambda \in \mathfrak{g}$, and  
\begin{center}
$\omega_{O(\Lambda)}({\text{ad}}(X)Z,{\text{ad}}(Y)Z) = -(Z,[X,Y])$, \ \ \ $\forall X,Y \in \mathfrak{g}$.
\end{center} 
In the above setting, the adjoint action on $O(\Lambda)$ defines a Hamiltonian action and, from the isomorphism $\mathfrak{g}\simeq \mathfrak{g}^\ast$, we have an equivariant moment map $\Phi \colon O(\Lambda) \to \mathfrak{g}^{\ast}$ defined by the natural inclusion map $O(\Lambda) \hookrightarrow \mathfrak{g}$. 

\begin{remark}
In the setting of Hamiltonian $G$-spaces defined by (co)adjoint orbits, throughout this work we shall assume that $G$ is compact. Therefore, the identification described in the above remark will be used without many explanations. Also, we shall assume some basic results about the theory of compact Lie groups, e.g. \cite{Sepanski}.
\end{remark}
%\begin{center}
%$O(\Lambda) = \Big \{ {\text{Ad}}(g)\Lambda \in \mathfrak{g} \ \Big | \ g \in G \Big \} $ \ \ \ and \ \ \ $\omega_{O(\Lambda)}({\text{ad}}(X)Z,{\text{ad}}(Y)Z) = -(Z,[X,Y]),$
%\end{center}
%$\forall X,Y \in \mathfrak{g}$. It is straightforward to verify that the adjoint action of $G$ on $O(\Lambda)$ is a Hamiltonian action with equivariant moment map $\Phi \colon O(\Lambda) \to \mathfrak{g}^{\ast} \cong \mathfrak{g}$, given by the inclusion map $O(\Lambda) \hookrightarrow \mathfrak{g}$. 
\end{remark}

Let $\mathfrak{g}$ be the Lie algebra of a compact and connected Lie group $G$. We have a Poisson bracket $ \{ \cdot,\cdot  \}_{\mathfrak{g}^{\ast}}$ on the manifold $\mathfrak{g}^{\ast}$ defined as follows\footnote{In general, given a vector space $V$, there exists a correspondence between Lie algebra structures on $V$ and linear Poisson structures on $V^{\ast}$, see for instance \cite[p. 367]{Rudolph}.}. Given $F_{1}, F_{2} \in C^{\infty}(\mathfrak{g}^{\ast})$ and $\xi \in \mathfrak{g}^{\ast}$, we set
$$\big \{ F_{1},F_{2}\big \}_{\mathfrak{g}^{\ast}}(\xi) = - \big \langle \xi , \big [(dF_{1})_{\xi},(dF_{2})_{\xi} \big ] \big \rangle,$$
where we see $(dF_{i})_{\xi}$, $i = 1,2$, as elements of $\mathfrak{g}$ via the natural identification $T_{\xi}^{\ast}\mathfrak{g}^{\ast} \cong \mathfrak{g}$. Also, for a fixed $\xi \in \mathfrak{g}^{\ast}$, it will be convenient to denote by $\nabla F(\xi)$ the element of $\mathfrak{g}$ which satisfies the pairing
$$
(dF)_{\xi}(\eta) = \big \langle \eta , \nabla F(\xi)  \big \rangle,
$$
for every $F \in  C^{\infty}(\mathfrak{g}^{\ast})$ and every $\eta \in T_{\xi}\mathfrak{g}^{\ast}$. From this, we can rewrite the previous expression of $ \{ \cdot,\cdot \}_{\mathfrak{g}^{\ast}}$ as follows:
$$\big \{ F_{1},F_{2}\big \}_{\mathfrak{g}^{\ast}}(\xi) = - \big \langle \xi ,\big [\nabla F_{1}(\xi),\nabla F_{2}(\xi) \big ] \big \rangle.$$
With the bracket above, the pair $(\mathfrak{g}^{\ast}, \{ \cdot,\cdot \}_{\mathfrak{g}^{\ast}})$ defines a Poisson manifold, e.g. \cite[Example 1.1.3]{QUANTUM}. 

Given a Hamiltonian system $(M,\omega,H)$, denoting by $\big \{ \cdot, \cdot \big \}_{M}$ the Poisson structure induced by the symplectic structure $\omega$, we are interested in the following concept of integrability.

\begin{definition}[Liouville integrability]
Let $(M,\omega,H)$ be a Hamiltonian system. We say that such a system is integrable if there exist $H_{1}, \ldots, H_{n} \colon (M,\omega) \to \mathbb{R}$, such that $H_{i} \in C^{\infty}(M)$, for each $i = 1,\ldots,n = \frac{1}{2}\dim_{\mathbb{R}}(M)$, with $H = H_{1}$, satisfying 

\begin{itemize}

    \item $\big \{H_{i},H_{j} \big \}_{M} = 0$, for all $i,j = 1,\ldots,n$,
    
    \item $dH_{1} \wedge \cdots \wedge dH_{n} \neq 0 $, in an open dense subset of $M$.
    
\end{itemize}

\end{definition}

\begin{remark}
In the setting of the above definition, by considering the induced map $\mathscr{H} \colon (M,\omega) \to \mathbb{R}^{n}$, such that $\mathscr{H} = (H_{1}, \ldots, H_{n})$, we denote an integrable system by $(M,\omega,\mathscr{H})$.
\end{remark}

From the first item of the integrability condition described above, in order to study integrable systems it will be useful to consider the following concept. 

\begin{definition}
Let $(M, \{ \cdot,\cdot \}_{M})$ be a Poisson manifold. A smooth function $C \in C^{\infty}(M)$, is called a Casimir function if it satisfies
$$\big \{ C,F \big \}_{M} = 0,$$
for every $F \in C^{\infty}(M)$.
\end{definition}

\begin{example}
Consider the Poisson manifold $(\mathfrak{g}^{\ast}, \{ \cdot,\cdot \}_{\mathfrak{g}^{\ast}})$ described previously. In this particular case, we have that the Casimir functions of $(\mathfrak{g}^{\ast}, \{ \cdot,\cdot \}_{\mathfrak{g}^{\ast}})$ are exactly the ${\text{Ad}}^{\ast}$-invariant functions. For a more general discussion about Casimir functions with respect to the Lie-Poisson bracket, see for instance \cite[p. 463]{INTROMECH}.   
\end{example}
\begin{remark}
In what follows, given $(\mathfrak{g}^{\ast}, \{ \cdot,\cdot \}_{\mathfrak{g}^{\ast}})$, we shall denote by $C^{\infty}(\mathfrak{g}^{\ast})^{{\text{Ad}^{\ast}}}$ the subalgebra of  ${\text{Ad}}^{\ast}$-invariant functions defined on $\mathfrak{g}^{\ast}$.
\end{remark}
In this work we are interested in studying Hamiltonian systems defined by the following special class of functions.
\begin{definition}
Let $(M,\omega,G,\Phi)$ be a Hamiltonian $G$-space. A collective Hamiltonian on $M$ is a function of the form $H = \Phi^{\ast}(F) \in C^{\infty}(M)$, where $F \in  C^{\infty}(\mathfrak{g}^{\ast})$.
\end{definition}
Now we will provide an expression for the Hamiltonian vector field $X_{H} \in \mathfrak{X}(M)$, associated to a collective Hamiltonian $H = \Phi^{\ast}(F) \in C^{\infty}(M)$, for more details see \cite[p. 241]{COL}.

At first, note that by fixing a basis $\{X_{i}\}$ for $\mathfrak{g}$ and denoting by $\{X_{i}^{\ast}\}$ its dual, we have 
$$\Phi = \displaystyle \sum_{i}\Phi^{i}X_{i}^{\ast}, \ \ \mbox{  and  } \ \ D\Phi = \displaystyle \sum_{i}d \langle \Phi,X_{i} \rangle X_{i}^{\ast},$$
where each component function $\Phi^{i} = \langle \Phi,X_{i} \rangle$ satisfies the equation $d \langle \Phi,X_{i} \rangle + \iota_{\delta \tau(X_{i})} \omega = 0.$
Recall that $\delta \tau$ denotes the infinitesimal action associated to the Hamiltonian action $ \tau \colon G \to {\text{Diff}}(M)$. Therefore, given $H = \Phi^{\ast}(F) \in C^{\infty}(M)$, we have
$$dH = dF \circ D\Phi = dF(\displaystyle \sum_{i}d \langle \Phi,X_{i} \rangle X_{i}^{\ast}) = - \displaystyle \sum_{i} \langle X_{i}^{\ast}, (\nabla F) \circ \Phi \rangle \iota_{\delta \tau(X_{i})} \omega = - \iota_{\delta \tau((\nabla F) \circ \Phi)} \omega.$$
Hence, we obtain
\begin{equation}
\label{Hvectorcollective}
X_{H} = \delta \tau((\nabla F) \circ \Phi).
\end{equation}

Let us illustrate the ideas above by means of an example which will be useful afterwards in this work. Further discussions about the Hamiltonian flow of collective Hamiltonians can be found in \cite[p. 241-242]{COL}.
\begin{example}
Consider now the Hamiltonian $G$-space $(O(\lambda),\omega_{O(\lambda)},G,\Phi)$. If we take a collective Hamiltonian $H = \Phi^{\ast}(F) \in  C^{\infty}(O(\lambda))$, from Eq. \ref{Hvectorcollective} we have $X_{H} = {\text{ad}}^{\ast}((\nabla F) \circ \Phi).$ Thus, since $\Phi$ is just the inclusion map $O(\lambda) \hookrightarrow \mathfrak{g}^{\ast}$, it follows that the dynamical system defined by $X_{H}$, such that $H = \Phi^{\ast}(F) \in  C^{\infty}(O(\lambda))$, can be understood through the Hamiltonian system $(\mathfrak{g}^{\ast},\{\cdot,\cdot\}_{\mathfrak{g}^{\ast}},F)$.
\end{example}

%%%%%%%%%%%%%%%%%%%%%%%%%%%%%%%%%%%%%%%%%%%%%%%%%%%%%%%%%%%%%%%%%%%%%%%%%%%%%%%%%%%%%%%%%%%%%%%%%%%%%%%%%%%%%%%%%%%%%%%%%%%%%%%%%%%%%%%%%%%%%%%%%%%%%%%%%%%%%%%%%%%%%%%%
\subsection{Thimm's trick and Gelfand-Tsetlin integrable systems} \label{GT-system}
%\setcounter{equation}{0}%

%%%%%%%%%%%%%%%%%%%%%%%%%%%%%%%%%%%%%%%%%%%%%%%%%%%%%%%%%%%%%%%%%%%%%%%%%%%%%%%%%%%%%%%%%%%%%%%%%%%%%%%%%%%%%%%%%%%%%%%%%%%%%%%%%%%%%%%%%%%%%%%%%%%%%%%%%%%%%%%%%%%%%%%
Now we will describe how to obtain conserved quantities in involution when we consider Hamiltonian systems defined by collective Hamiltonians. In what follows, unless otherwise stated, given  a Hamiltonian $G$-space $(M,\omega,G,\Phi)$, we shall assume the equivariance condition for the moment map $\Phi \colon (M,\omega) \to \mathfrak{g}^{\ast}$. Under this assumption, we have that the map $\Phi \colon (M,\{\cdot,\cdot\}_{M}) \to (\mathfrak{g}^{\ast},\{\cdot,\cdot\}_{\mathfrak{g}^{\ast}})$ is a Poisson map, see for instance \cite[p. 497]{Rudolph}), so we obtain
$$\big \{F \circ \Phi , I \circ \Phi \big \}_{M}(p) = \big \{F , I \big \}_{\mathfrak{g}^{\ast}}(\Phi(p)), $$
for every $p \in M$ and $F,I \in C^{\infty}(\mathfrak{g}^{\ast})$. Therefore, from Chevalley's theorem one can find ${\text{rank}}(G)$ independent Poisson commuting functions from the Casimir functions of $(\mathfrak{g}^{\ast}, \{ \cdot,\cdot \}_{\mathfrak{g}^{\ast}})$. However, it is often the case that ${\text{rank}}(G) < \frac{1}{2}\dim_{\mathbb{R}}(M)$. In order to find additional independent Poisson commuting functions, one can proceed following the ideas of the well-known Thimm's trick \cite[Proposition 4.1]{THIMM}. More precisely, if we consider a closed and connected Lie subgroup $K \subset G$, we have a natural Hamiltonian action of $K$ on $(M,\omega)$ induced by restriction. Thus, we obtain a Hamiltonian $K$-space $(M,\omega,K,\Phi_{K})$, where the moment map $\Phi_{K} \colon (M,\omega) \to \mathfrak{k}^{\ast} = {\text{Lie}}(K)^{\ast}$, is given by 
$$\Phi_{K} = \pi_{K} \circ \Phi,$$
such that $\pi_{K} \colon \mathfrak{g}^{\ast} \to \mathfrak{k}^{\ast}$ is the projection induced by the inclusion $\mathfrak{k} \hookrightarrow \mathfrak{g}$. From this, if we take two collective Hamiltonians $\Phi^{\ast}(F),\Phi_{K}^{\ast}(I) \in C^{\infty}(M)$, we obtain 
$$\big \{F \circ \Phi , I \circ \Phi_{K} \big \}_{M} = \big \{F \circ \Phi , I \circ \pi_{K} \circ \Phi \big \}_{M} =  \big \{F , I \circ \pi_{K} \big \}_{\mathfrak{g}^{\ast}} \circ \Phi.$$
Thus, all collective Hamiltonians obtained from the Casimir functions of $(\mathfrak{g}^{\ast},\{\cdot,\cdot\}_{\mathfrak{g}^{\ast}})$ and $(\mathfrak{k}^{\ast},\{\cdot,\cdot\}_{\mathfrak{k}^{\ast}})$ are conserved quantities in involution for the Hamiltonian system $(M,\omega,\Phi_{K}^{\ast}(I))$, e.g. \cite{THIMM}, \cite{COL}, \cite{GELFORB}. As an application of the above ideas, we obtain the following general construction:
%\begin{proposition}[\cite{THIMM}, \cite{COL}, \cite{GELFORB}]
%\label{ThimmThimm}
%Let $(M,\omega,G,\Phi)$ be a Hamiltonian $G$-space, and let $K \subset G$ be a closed and connected Lie subgroup. If we consider the Hamiltonian system $(M,\omega, \Phi_{K}^{\ast}(I))$, for some $I \in C^{\infty}(\mathfrak{k}^{\ast})$, then all collective Hamiltonians obtained from the Casimir functions of $(\mathfrak{g}^{\ast},\{\cdot,\cdot\}_{\mathfrak{g}^{\ast}})$ and $(\mathfrak{k}^{\ast},\{\cdot,\cdot\}_{\mathfrak{k}^{\ast}})$ are quantities in involution for the Hamiltonian system $(M,\omega,\Phi_{K}^{\ast}(I))$.
%\end{proposition}

\begin{itemize}

\item Consider the Hamiltonian $G$-space $(O(\Lambda),\omega_{O(\Lambda)},G,\Phi)$ as in Remark \ref{Adjointorbit}. Given a nested chain of closed and connected subgroups $G = K_{0} \supset K_{1} \supset \cdots \supset K_{s}$, associated to each Hamiltonian $K_{i}$-space  $(O(\Lambda),\omega_{O(\Lambda)},K_{i},\Phi_{i})$, we have $\Phi_{i} = \pi_{i}^{(l)} \circ \Phi_{l}$, where $\pi_{i}^{(l)} \colon \mathfrak{k}_{l} \to \mathfrak{k}_{i}$ denotes the projection induced by the inclusion $\mathfrak{k}_{i} \hookrightarrow \mathfrak{k}_{l}$, $0 \leq l < i \leq s$.

\item From the above data, if one considers a Hamiltonian system $(O(\Lambda),\omega_{O(\Lambda)},\Phi_{s}^{\ast}(I))$, for some $I \in C^{\infty}(\mathfrak{k}_{s})$, applying Thimm's trick iteratively we obtain the following set of conserved quantities in involution 
\begin{equation}
\label{GTS}
H^{(i)}_{j} = \Phi_{i}^{\ast}(I_{j}^{(i)}), \ \ \ I_{j}^{(i)} \in C^{\infty}(\mathfrak{k}_{i})^{{\text{Ad}}},
\end{equation}
where $1 \leq j \leq r_{i}$, $r_{i} = {\text{rank}}(K_{i})$, $1 \leq i \leq s$. 
\end{itemize}

When the integrability condition holds\footnote{For more details we suggest \cite{Panyushev}.} for the set of conserved quantities in involution described above, the integrable system is called \textit{Gelfand-Tsetlin integrable system} \cite{GELFORB}. In concrete cases it is convenient to work with the action coordinate of the Gelfand-Tsetlin integrable system. In order to obtain theses action coordinates one can proceed as follows: Let $T \subset G$ be a Cartan subgroup and let $G = T_{0} \supset T_{1} \supset \cdots \supset T_{s}$ be a sequence of subgroups of $T$ such that $T_{i} \subset K_{i}$ is a Cartan subgroup, $i = 1,\ldots s$. Now by taking a positive Weyl chamber $(\mathfrak{t}_{i})_{+} \subset \mathfrak{t}_{i}$, for all $i = 1,\ldots,s$, one can consider the (continuous) \textit{sweeping map} 
\begin{equation}
s_{i} \colon \mathfrak{k}_{i} \to (\mathfrak{t}_{i})_{+}, 
\end{equation}
which is defined by letting $s_{i}(X)$ be the unique element of the set $({\text{Ad}}(K_{i})X) \cap (\mathfrak{t}_{i})_{+}$. For each $i = 1,\ldots,s$, chose a basis $\xi_{i1}, \ldots,\xi_{ir_{i}}$ of the integer lattice in $\mathfrak{t}_{i}$ and define
\begin{equation}
\label{GSAction}
\Lambda_{ij}:= l_{\xi_{ij}}^{+} \circ s_{i} \circ \Phi_{i},
\end{equation}
where $l_{\xi_{ij}}^{+}$ is the restriction to $(\mathfrak{t}_{i})_{+}$ of the linear functional $l_{\xi_{ij}} = (-,\xi_{ij})$, here we consider some fixed ${\rm{Ad}}$-invariant inner product $(-,-)$ on $\mathfrak{g}$. The set of functions $\{\Lambda_{ij}\}$ defines {\textit{Guillemin and Sternberg’s action coordinates}} on $O(\Lambda)$, see for instance \cite[p. 119]{GELFORB} and \cite{Lane}. We also shall refer to the set of functions $\{\Lambda_{ij}\}$ as Gelfand-Tsetlin integrable system. Further, a more detailed exposition about the above construction for the particular case of adjoint orbits of ${\rm{U}}(n)$ will be carried out afterward in Section \ref{GTU(n)}. 

\section{Lax pairs and spectral equation}

%%%%%%%%%%%%%%%%%%%%%%%%%%%%%%%%%%%%%%%%%%%%%%%%%%%%%%%%%%%%%%%%%%%%%%%%%%%%%%%%%%%%%%%%%%%%%%%%%%%%%%%%%%%%%%%%%%%%%%%%%%%%%%%%%%%%%%%%%%%%%%%%%%%%%%%%%%%%%%%%%%%%%%%%%
In this section, we will introduce some basic ideas about the concept of Lax pairs and  describe their relation with the study of integrability in the context of Hamiltonian systems. More details about this topic can be found in \cite{INTCLASSICALSYSTEM}, and \cite[p. 578]{Rudolph}.

%%%%%%%%%%%%%%%%%%%%%%%%%%%%%%%%%%%%%%%%%%%%%%%%%%%%%%%%%%%%%%%%%%%%%%%%%%%%%%%%%%%%%%%%%%%%%%%%%%%%%%%%%%%%%%%%%%%%%%%%%%%%%%%%%%%%%%%%%%%%%%%%%%%%%%%%%%%%%%%%%%%%%%

%%%%%%%%%%%%%%%%%%%%%%%%%%%%%%%%%%%%%%%%%%%%%%%%%%%%%%%%%%%%%%%%%%%%%%%%%%%%%%%%%%%%%%%%%%%%%%%%%%%%%%%%%%%%%%%%%%%%%%%%%%%%%%%%%%%%%%%%%%%%%%%%%%%%%%%%%%%%%%%%%%%%%%%

\begin{definition}
A Lax pair for a Hamiltonian system $(M,\omega,H)$ is defined by a pair of matrix-valued smooth functions $L,P \colon (M,\omega) \to \mathfrak{gl}(r,\mathbb{R})$, such that the
equation of motion associated to $H \in  C^{\infty}(M)$ is equivalent to the equation 
\begin{equation}\label{lax-eq}
 \frac{d}{dt}L + \big [L,P \big] = 0.
\end{equation}
\end{definition}
\begin{remark}
We observe that the derivative in the definition above is taken when we consider the composition of $L$ with the Hamiltonian flow of $X_{H} \in \mathfrak{X}(M)$. 
\end{remark}
The advantage of dealing with Eq. \ref{lax-eq} instead of the equation of motion induced by $H \in  C^{\infty}(M)$ is that it can be easily solved. Actually, if we consider the initial value problem 
$$\displaystyle \frac{d}{dt}L = \big [ P,L \big], \ \ \ \mbox{   with   } \ \ \  L(0) = L_{0},$$
its solution is given by $L(t) = g(t)L_{0}g(t)^{-1}$, where $g \colon (-\epsilon,\epsilon) \to {\rm{GL}}(r,\mathbb{R})$ is determined by the initial value problem 
$$\displaystyle \frac{dg}{dt} = P(t)g(t), \ \ \ \mbox{    with   } \ \ \ g(0) = \mathds{1}.$$

\begin{example}[Harmonic Oscillator] A basic example to illustrate the previous discussion is provided by the harmonic oscillator. Consider the Hamiltonian system $(\mathbb{R}^{2},dp\wedge dq, H)$, where the Hamiltonian function is given by
\begin{equation}
H(q,p) = \displaystyle \frac{1}{2}\big (p^{2} + C^{2}q^{2} \big ),
\end{equation}
such that $C \in \mathbb{R} \backslash \{0\}$. A straightforward computation shows that
$$dH + \iota_{X_{H}} (dp\wedge dq) = 0 \iff X_{H} = p \partial_{q} - C^{2}q\partial_{p}.$$
From this, we obtain the following equations of motion
$$\displaystyle \frac{dq}{dt} = p, \ \ \  \mbox{  and  } \ \ \ \displaystyle \frac{dp}{dt} = -C^{2} q.$$
We have a Lax pair $L,P \colon (\mathbb{R}^{2},dp\wedge dq) \to \mathfrak{gl}(2,\mathbb{R})$ for the Hamiltonian system $(\mathbb{R}^{2},dp\wedge dq, H)$ defined by
$$L = \begin{pmatrix}
 p & \ C q \\
C q & -p
\end{pmatrix}, \ \ \  \mbox{   and   } \ \ \ P = \displaystyle \frac{1}{2}\begin{pmatrix}
 0 & - C\\
C & \ \ 0
\end{pmatrix}.$$
In fact, from a straightforward computation one can check that 
$$
\displaystyle \frac{d}{dt}L + \big [L,P \big] = 0 \iff \begin{cases}
    \frac{dq}{dt} = p, \\
    \frac{dp}{dt} = -C^{2}q.                    \\
  \end{cases}$$
Notice that $H(q,p) = -\displaystyle \frac{1}{2}\det(L) = \displaystyle \frac{1}{4}{\text{Tr}}(L^{2}).$
\end{example}

The key point which makes the existence of Lax pairs an important tool in the study of a Hamiltonian systems is the following. Suppose we have a Lax pair $(L,P)$ for a Hamiltonian system $(M,\omega,H)$. If we consider a smooth function $F \colon \mathfrak{gl}(r,\mathbb{R}) \to \mathbb{R}$ which is invariant by the adjoint action, i.e.
$$F(gXg^{-1}) = F(X), \mbox{   for all  } g \in {\rm{GL}}(r,\mathbb{R}), \mbox{  and  } X \in \mathfrak{gl}(r,\mathbb{R}), $$
and consider the composition $I = F \circ L \in C^{\infty}(M)$. Then, we obtain a function which is constant over the Hamiltonian flow of $X_{H} \in \mathfrak{X}(M)$. In fact, we have
$$I(t) = F(L(t)) = F( g(t)L_{0}g(t)^{-1}) = F(L_{0}) = {\text{constant}}.$$
Therefore, 
$$\big \{ H,I \big \}_{M} = X_{H}(I) = 0.$$
Hence, one can obtain conserved quantities from the procedure described above. More precisely, if we consider a Hamiltonian system $(M,\omega,H)$ which admits a Lax pair $L,P \colon M \to \mathfrak{gl}(r,\mathbb{R})$, then the coefficients of the characteristic polynomial of $L$, namely,
\begin{equation}
\label{Characteristicpoly}
\det(w\mathds{1}_{r} - L) = f_{r}(L)w^{r} + f_{r-1}(L)w^{r-1} + \ldots + f_{1}(L)w + f_{0}(L),
\end{equation}
provide a set of conserved quantities for the Hamiltonian system $(M,\omega,H)$. Moreover, if $L = U\Lambda U^{-1}$, where $U \colon (M,\omega) \to {\rm{GL}}(r,\mathbb{R})$ and $\Lambda \colon M \to \mathfrak{gl}(r,\mathbb{R})$ is a diagonal matrix of the form
$$\Lambda = {\text{diag}}(\Lambda_{1}, \ldots, \Lambda_{r}),$$
it follows that the functions defined by the eigenvalues of $L$ are conserved quantities for the Hamiltonian system $(M,\omega,H)$. The equation 
\begin{equation}
\det(w\mathds{1}_{r} - L) = 0,
\end{equation}
is called the {\textit{spectral equation}} associated to $L$, and we refer to the constants of motion obtained from the characteristic polynomial of $L$ as {\textit{spectral invariants}} of $L$. 

\begin{remark}
It is worth mentioning that the involution property for the eigenvalue of $L$ is equivalent to the existence of a $r$-matrix on the phase space, see for instance \cite[\S 2.5]{INTCLASSICALSYSTEM}.
\end{remark}

%%%%%%%%%%%%%%%%%%%%%%%%%%%%%%%%%%%%%%%%%%%%%%%%%%%%%%%%%%%%%%%%%%%%%%%%%%%%%%%%%%%%%%%%%%%%%%%%%%%%%%%%%%%%%%%%%%%%%%%%%%%%%%%%%%%%%%%%%%%%%%%%%%%%%%%%%%%%%%%%%%%%%%%%%

\section{A Lax pair formalism for Gelfand-Tsetlin integrable systems}

In this section, we prove our first result (Theorem \ref{Maintheorem}), which consists of providing a Lax pair formulation for the Gelfand-Tsetlin integrable system. In order to do this, we reformulate the construction described in Section \ref{GT-system} purely in terms of Lax equations, see for instance Theorem \ref{prop43} and Corollary \ref{C5S5.2Teo5.2.4}.

\subsection{Lax equation and collective Hamiltonians}

Let us start by describing the relationship between collective Hamiltonians and the Lax equation\footnote{For more details we suggest \cite{Perelomov} and references therein.}. As we have seen in the Section \ref{GT-system}, the functions which compose Gelfand-Tsetlin systems are given by collective Hamiltonians, i.e., if we consider a Hamiltonian $G$-space $(M,\omega, G,\Phi)$, we can take $F \in C^{\infty}(\mathfrak{g}^{\ast})$ and consider the smooth function given by 
$$\Phi^{\ast}(F) = F \circ \Phi \colon M \to \mathbb{R}.$$
The Hamiltonian vector field associated to a function defined as above has the following expression
$$X_{\Phi^{\ast}(F)}(p) = \delta \tau(\nabla F(\Phi(p)))_{p},$$
for every $p \in M$, see for instance Eq.  \ref{Hvectorcollective}. If we consider the Hamiltonian $G$-space $(O(\lambda),\omega_{O(\lambda)},G,\Phi)$ and take a collective Hamiltonian $\Phi^{\ast}(F) \in  C^{\infty}(O(\lambda))$, we have 
$$X_{\Phi^{\ast}(F)}(\xi) = {\text{ad}}^{\ast}(\nabla F(\Phi(\xi))) \xi,$$
for every $\xi \in O(\lambda)$. By means of an ${\text{Ad}}$-invariant isomorphism  $\mathfrak{g}^{\ast} \cong \mathfrak{g}$ given by some ${\text{Ad}}$-invariant inner product, and the identification between coadjoint and adjoint orbits $O(\lambda) \cong O(\Lambda)$, we obtain from the ordinary differential equation associated to $X_{\Phi^{\ast}(F)}$ the following expression
$$\displaystyle \frac{d}{dt}\varphi_{t}(Z) = X_{\Phi^{\ast}(F)}(\varphi_{t}(Z)) = {\text{ad}}(\nabla F(\Phi(\varphi_{t}(Z))))\varphi_{t}(Z),$$
for every initial condition $Z \in O(\Lambda)$. Since the moment map $\Phi \colon O(\Lambda) \to \mathfrak{g}$ is just the inclusion map, we have the following equation for every $Z \in O(\Lambda)$
$$\displaystyle \frac{d}{dt}\Phi(\varphi_{t}(Z)) = \big [\nabla F(\Phi(\varphi_{t}(Z)), \Phi(\varphi_{t}(Z)) \big ].$$
Notice that, if we denote $X = \nabla F(\Phi(\varphi_{t}(Z))$ and $Y = \varphi_{t}(Z)$, we have
$$\displaystyle \frac{d}{dt}\Phi(\varphi_{t}(Z)) = (D\Phi)_{Y}({\text{ad}}(X)Y) = {\text{ad}}(X)\Phi(Y) = \big [X, \Phi(Y) \big].$$
From this, we have the following proposition:
\begin{proposition}
\label{C5S5.1P5.1.1}
Given the Hamiltonian $G$-space $(O(\Lambda),\omega_{O(\Lambda)},G,\Phi)$, the dynamic associated to any collective Hamiltonian is completely determined by a zero curvature equation
\begin{equation}
\label{EulerLax}
\displaystyle \frac{d}{dt}L + \big [L,P \big] = 0,
\end{equation}
where $L,P \colon (O(\Lambda),\omega_{O(\Lambda)}) \to \mathfrak{g}$ are Lie algebra-valued smooth functions.
\end{proposition}
\begin{proof}
Let $\Phi^{\ast}(F) \in C^{\infty}(O(\Lambda))$ be a collective Hamiltonian associated to some $F \in C^{\infty}(O(\Lambda))$. From the Hamiltonian flow of $X_{\Phi^{\ast}}(F) \in \mathfrak{X}(O(\Lambda))$, we have the following O.D.E.
\begin{equation}
\label{ODE}
\displaystyle \frac{d}{dt}\varphi_{t}(Z) =  X_{\Phi^{\ast}(F)}(\varphi_{t}(Z)) = {\text{ad}}(\nabla F(\Phi(\varphi_{t}(Z))))\varphi_{t}(Z),
\end{equation}
for every $Z \in O(\Lambda)$. We define the following pair of Lie algebra-valued functions
\begin{center}
$L  \colon Z \in O(\Lambda) \mapsto \Phi(Z) \in \mathfrak{g},$ \ \ and \ \  
$P \colon Z \in O(\Lambda) \mapsto \nabla F(\Phi(Z)) \in \mathfrak{g}.$
\end{center}
From the previous comments, we obtain 
\begin{align}
\displaystyle \frac{d}{dt}L(\varphi_{t}(Z)) = \displaystyle \frac{d}{dt}\Phi(\varphi_{t}(Z)) = \big [\nabla F(\Phi(\varphi_{t}(Z))), \Phi(\varphi_{t}(Z)) \big ].
\end{align}
Hence, from the definition of $L$ and $P$ we have that Eq. \ref{ODE} is exactly the equation $\frac{d}{dt}L + [L,P] = 0$.
\end{proof}

\begin{remark}
Notice that, since the adjoint action of a compact Lie group $G$ on its Lie algebra is a proper action, we have that $O(\Lambda) \subset \mathfrak{g}$ is a compact embedded submanifold of $\mathfrak{g}$. In particular, it follows that $C^{\infty}(O(\Lambda)) = \Phi^{\ast}(C^{\infty}(\mathfrak{g}))$, see for instance \cite[p. 29]{WARNER}. Thus, given $\psi \in C^{\infty}(O(\Lambda))$, we have some $I \in C^{\infty}(\mathfrak{g})$, such that $\psi = \Phi^{\ast}(I)$. Therefore, given a Hamiltonian system $(O(\Lambda),\omega_{O(\Lambda)},\Phi^{\ast}(F))$, for some $F \in C^{\infty}(O(\Lambda))$, a straightforward computation shows that
\begin{align}
\displaystyle\frac{d}{dt}\psi(\varphi_{t}(Z)) =  \big (\nabla I (L(\varphi_{t}(Z))), \frac{d}{dt}L(\varphi_{t}(Z)) \big ),
\end{align}
for every $Z \in O(\Lambda)$, where $L = \Phi$ and $P = \nabla F(\Phi)$. Hence, the action of the Hamiltonian vector field $X_{\Phi^{\ast}(F)} \in \mathfrak{X}(O(\Lambda))$ on every $\psi \in C^{\infty}(O(\Lambda))$ is completely determined by the Lax equation
$$\displaystyle \frac{d}{dt}L(\varphi_{t}(Z)) + \big [L(\varphi_{t}(Z)), P(\varphi_{t}(Z)) \big ] = 0.$$
\end{remark}

\begin{remark}[Euler's equation] It is worth pointing out that Eq.  \ref{EulerLax} is a generalization of the well-known Euler's equation of rotation of a rigid body about a fixed point \cite{Euler}. Rotation of the $n$-dimensional rigid body is also described by this type of equation as was first shown by Arnold, see for instance \cite{Arnold1}, \cite{Arnold2}.

\end{remark}

%%%%%%%%%%%%%%%%%%%%%%%%%%%%%%%%%%%%%%%%%%%%%%%%%%%%%%%%%%%%%%%%%%%%%%%%%%%%%%%%%%%%%%%%%%%%%%%%%%%%%%%%%%%%%%%%%%%%%%%%%%%%%%%%%%%%%%%%%%%%%%%%%%%%%%%%%%%%%%%%%%%%%%
%%%%%%%%%%%%%%%%%%%%%%%%%%%%%%%%%%%%%%%%%%%%%%%%%%%%%%%%%%%%%%%%%%%%%%%%%%%%%%%%%%%%%%%%%%%%%%%%%%%%%%%%%%%%%%%%%%%%%%%%%%%%%%%%%%%%%%%%%%%%%%%%%%%%%%%%%%%%%%%%%%%%%%

%%%%%%%%%%%%%%%%%%%%%%%%%%%%%%%%%%%%%%%%%%%%%%%%%%%%%%%%%%%%%%%%%%%%%%%%%%%%%%%%%%%%%%%%%%%%%%%%%%%%%%%%%%%%%%%%%%%%%%%%%%%%%%%%%%%%%%%%%%%%%%%%%%%%%%%%%%%%%%%%%%%%%%

\subsection{Thimm's trick and spectral invariants}
%\setcounter{equation}{0}%
%%%%%%%%%%%%%%%%%%%%%%%%%%%%%%%%%%%%%%%%%%%%%%%%%%%%%%%%%%%%%%%%%%%%%%%%%%%%%%%%%%%%%%%%%%%%%%%%%%%%%%%%%%%%%%%%%%%%%%%%%%%%%%%%%%%%%%%%%%%%%%%%%%%%%%%%%%%%%%%%%%%%%%

%Once we have described the dynamic associated to any collective Hamiltonian in terms of a Lax equation, our next task will be understanding how one can use this approach to recover the quantities in involution obtained by means of Thimm's trick.

In what follows, we investigate the relationship between Thimm's trick and the Lax equation \ref{EulerLax}. Given a Hamiltonian $G$-space $(O(\Lambda),\omega_{O(\Lambda)},G,\Phi)$, let $K \subset G$ be a closed connected Lie subgroup of $G$. By restriction, one can consider the Hamiltonian $K$-space $(O(\Lambda),\omega_{O(\Lambda)},K,\Phi_{K})$, where
$$\Phi_{K} \colon O(\Lambda) \to \mathfrak{k}, \ \ \ \mbox{  such that  } \ \ \ \Phi_{K} = \pi_{K} \circ \Phi.$$
Here we denote by $\pi_{K} \colon \mathfrak{g} \to \mathfrak{k}$ the orthogonal projection map. By taking $F \in C^{\infty}(\mathfrak{k})$, we consider the collective Hamiltonian $\Phi_{K}^{\ast}(F) \in C^{\infty}(O(\Lambda))$. Note that 
$$\Phi_{K}(F) = \Phi^{\ast}(F \circ \pi_{K}).$$
From this, we denote $\widetilde{F} = F \circ \pi_{K}$ and consider $\Phi_{K}^{\ast}(F) = \Phi^{\ast}(\widetilde{F})$ also as a collective Hamiltonian associated to the Hamiltonian $G$-space $(O(\Lambda),\omega_{O(\Lambda)},G,\Phi)$. As we have seen in the previous section, the dynamic associated to $\Phi^{\ast}(\widetilde{F})$ is completely determined by the Lax equation 
$$\displaystyle \frac{d}{dt}L(\varphi_{t}(Z)) + \big [L(\varphi_{t}(Z)), P(\varphi_{t}(Z)) \big ] = 0,$$
for every $Z \in O(\Lambda)$, where $L = \Phi$ and $P = \nabla \widetilde{F}(\Phi)$. Now we consider the following result.

\begin{lemma}
\label{C5S5.2L5.2.1}
In the setting above, we have $(\nabla F) \circ \pi_{K} = \nabla (F \circ \pi_{K})$, $\forall F \in  C^{\infty}(\mathfrak{k})$.
\end{lemma}
\begin{proof}
It follows directly from the fact that $(D\pi_{K})_{Z} = \pi_{K}$, for all $Z \in \mathfrak{g}$. 
%Denoting $\widetilde{F} = F \circ \pi_{K}$, since $(d\widetilde{F})_{Z} = (dF)_{\pi_{K}(Z)} \circ (D\pi_{K})_{Z}$, for every $Z \in \mathfrak{g}$, and $(D\pi_{K})_{Z} = \pi_{K}$, it follows that 
%\begin{center}
%$(d\widetilde{F})_{Z}(Y) = (dF)_{\pi_{K}(Z)}(\pi_{K}(Y)) = \big ( \pi_{K}(Y) , \nabla F(\pi_{K}(Z)) \big ),$
%\end{center}
%for every $Y \in \mathfrak{g}$. On the other hand, we have 
%\begin{center}
%$(d\widetilde{F})_{Z}(Y) = (d\widetilde{F})_{Z}(\pi_{K}(Y)) = \big ( \pi_{K}(Y) , \nabla \widetilde{F}(Z) \big )$, 
%\end{center}
%for every $Y \in \mathfrak{g}$. Therefore, we obtain  $(\nabla F) \circ \pi_{K} = \nabla (F \circ \pi_{K})$, $\forall F \in  C^{\infty}(\mathfrak{k})$.
\end{proof}

From the result above, we see that for the Lax pair $L = \Phi$ and $P = \nabla \widetilde{F}(\Phi)$ associated to $\widetilde{F} = F \circ \pi_{K}$, since $\Phi_{K} = \pi_{K} \circ \Phi$, we have
$$P = (\nabla \widetilde{F}) \circ \Phi = (\nabla F) \circ \Phi_{K} \Longrightarrow  X_{\Phi^{\ast}(\widetilde{F})} = X_{\Phi_{K}^{\ast}(F)}.$$
Furthermore, we have the following equation
\begin{equation}
\label{laxrestriction}
\displaystyle \frac{d}{dt}\Phi_{K}(\varphi_{t}(Z)) = \big [\nabla F(\Phi_{K}(\varphi_{t}(Z))), \Phi_{K}(\varphi_{t}(Z)) \big ],
\end{equation}
where $\varphi_{t}(Z)$ is the Hamiltonian flow of $X_{\Phi^{\ast}(\widetilde{F})}$. In fact, if we denote $W = \varphi_{t}(Z) \in O(\Lambda)$ and $Y = \nabla F(\Phi_{K}(\varphi_{t}(Z))) \in \mathfrak{k}$, we have
\begin{equation}\label{eq-sec4-1}
\displaystyle \frac{d}{dt}\Phi_{K}(\varphi_{t}(Z)) = (D\Phi_{K})_{W}({\text{ad}}(Y)W) = {\text{ad}}(Y)\Phi_{K}(W).
\end{equation}
Note that the equality on the right-hand side of Eq. \ref{eq-sec4-1} follows from the fact that $\Phi_{K}$ is equivariant and $Y = \nabla F(\Phi_{K}(\varphi_{t}(Z))) \in \mathfrak{k}$. Now if we take $I \in C^{\infty}(\mathfrak{k})$ and consider the collective Hamiltonian 
$$ \psi = \Phi_{K}^{\ast}(I) \in C^{\infty}(O(\Lambda)),$$
from the equation of motion associated to the Hamiltonian system $(O(\Lambda),\omega_{O(\Lambda)},\Phi_{K}^{\ast}(F))$ we have
\begin{equation}
\label{eq-sec4-2}
\displaystyle \frac{d}{dt}\psi(\varphi_{t}(Z)) = \big \{ \Phi_{K}^{\ast}(F), \Phi_{K}^{\ast}(I)\big \}_{O(\Lambda)}(\varphi_{t}(Z)).
\end{equation}
The left-hand side of the above equation can be written as 
\begin{equation}
\label{Laxeqofmotion}
\displaystyle \frac{d}{dt}I(\Phi_{K}(\varphi_{t}(Z))) = (dI)_{\Phi_{K}(\varphi_{t}(Z))} \big (\displaystyle \frac{d}{dt}\Phi_{K}(\varphi_{t}(Z)) \big ) = \big ( \nabla I(\Phi_{K}(\varphi_{t}(Z))), \frac{d}{dt}\Phi_{K}(\varphi_{t}(Z))\big ).
\end{equation}
From this, we obtain the following theorem:
\begin{theorem}\label{prop43}
Given a Hamiltonian $G$-space $(O(\Lambda),\omega_{O(\Lambda)},G,\Phi)$, and given a closed and connected Lie subgroup $K \subset G$, for every $F \in C^{\infty}(\mathfrak{k})$, the Hamiltonian system $(O(\Lambda),\omega_{O(\Lambda)},\Phi_{K}^{\ast}(F))$ is completely determined by the Lax equation 
\begin{equation}
\frac{d}{dt}L_{K} + \big [L_{K}, P_{K}\big ] = 0,
\end{equation}
where $L_{K} = \Phi_{K}$, and $P_{K} = \nabla F(\Phi_{K})$. Moreover, for every finite dimensional unitary representation $\varrho \colon K \to {\rm{GL}}(V)$, the following facts hold:
\begin{enumerate}
\item The spectral equation
\begin{equation}
\label{spectral}
\det\big (w \mathds{1}_{V} - \varrho_{\ast}(L_{K})\big) = 0,
\end{equation}
is preserved by the Hamiltonian flow of $\Phi_{K}^{\ast}(F)$ through every $Z \in O(\Lambda)$;

\item The imaginary part of all eigenvalues of $ \varrho_{\ast}(L_{K})$ are conserved quantities in involution for the Hamiltonian system $(O(\Lambda),\omega_{O(\Lambda)},\Phi_{K}^{\ast}(F))$ on an open dense subset $\mathcal{U}_{K} \subset O(\Lambda)$.

\end{enumerate}

\end{theorem}
\begin{proof}
Consider $L_{K} = \Phi_{K}$ and $P_{K} = \nabla F(\Phi_{K})$. The first assertion follows directly from Lemma \ref{C5S5.2L5.2.1}, Eq. \ref{laxrestriction}, and Eq. \ref{Laxeqofmotion}. For the item 1, fix a maximal torus $T \subset K$, with Lie algebra $\mathfrak{t}$, take a closed positive Weyl chamber $\mathfrak{t}_{+} \subset \mathfrak{t}$, and consider the sweeping map $s_{K} \colon \mathfrak{k} \to \mathfrak{t}_{+}$, which is defined by letting $s_{K}(X)$ be the unique element of the set $({\text{Ad}}(K)X) \cap \mathfrak{t}_{+}$, $\forall X \in \mathfrak{k}$. Notice that $s_{K}$ is not smooth everywhere on $\mathfrak{k}$. From this, given $Z \in O(\Lambda)$, it follows that
\begin{center}
$L_{K}(Z) = {\text{Ad}}(k)s_{K}(L_{K}(Z))$, 
\end{center}
for some $k \in K$. Hence, given a finite dimensional unitary representation $\varrho \colon K \to {\rm{GL}}(V)$, by considering the corresponding induced Lie algebra representation $\varrho_{\ast} \colon \mathfrak{k} \to \mathfrak{gl}(V)$, since $K = \exp(\mathfrak{k})$, we obtain
\begin{center}
${\rm{Tr}}\big ([\varrho_{\ast}(L_{K}(Z))]^{m}\big ) = {\rm{Tr}}\big ( {\text{Ad}}(\varrho(k))[\varrho_{\ast}(s_{K}(L_{K}(Z)))]^{m}\big ) = {\rm{Tr}}\big ([\varrho_{\ast}(s_{K}(L_{K}(Z)))]^{m}\big )$,
\end{center}
for all $Z \in O(\Lambda)$ and for every $m \in \mathbb{N}$. By considering the Hamiltonian flow $\varphi_{t}(Z)$ of $\Phi_{K}^{\ast}(F)$ through some $Z \in O(\Lambda)$, it follows that
\begin{equation}
\label{flow}
\varphi_{t}(Z) = {\text{Ad}}(g(t))Z, \ \ \ (\forall t \in \mathbb{R})
\end{equation}
where $g \colon \mathbb{R} \to K$ is a solution of the initial value problem
$$\frac{dg}{dt} = v_{g(t)}, \ \ \mbox{   with   } \ \ g(0) = e,$$
for the vector field $v \in \mathfrak{X}(K)$, such that $v_{g} = (R_{g})_{\ast}(P_{K}({\text{Ad}}(g)Z))$, where $R_{g} \colon K \to K$ denotes the right translation, see for instance \cite[p. 241-242]{COL}. From this, since $L_{K}$ is equivariant and $s_{K}$ is ${\text{Ad}}$-invariant, given $Z \in O(\Lambda)$, it follows that 
\begin{center}
${\rm{Tr}}\big ([\varrho_{\ast}(s_{K}(L_{K}(\varphi_{t}(Z)))]^{m}\big) = {\rm{Tr}}\big ([\varrho_{\ast}(s_{K}(L_{K}(Z)))]^{m}\big ) \Longrightarrow \displaystyle \frac{d}{dt}\big \{{\rm{Tr}}\big ([\varrho_{\ast}(L_{K}(\varphi_{t}(Z)))]^{m}\big )\big\} = 0$.
\end{center}
Thus, once the coefficients of the characteristic polynomial $\det\big ( w \mathds{1}_{V} - \varrho_{\ast}(L_{K}(Z))\big)$ are given in terms of functions of the form ${\rm{Tr}}\big ([\varrho_{\ast}(L_{K}(Z))]^{m}\big )$, $m \in \mathbb{N}$, we obtain the item 1.

For item 2, we observe that associated to the Hamiltonian $K$-space $(O(\Lambda),\omega_{O(\Lambda)},K,\Phi_{K})$ we have the principal stratum $ \sigma_{K} \subset \mathfrak{t}_{+}$, which satisfies the property that 
\begin{center}
$\Phi_{K}(O(\Lambda)) \cap  \mathfrak{t}_{+} \subset \overline{\sigma_{K}}$ \ \ \ \ and \ \ \ $\Phi_{K}(O(\Lambda)) \cap \sigma_{K} \neq \emptyset$, 
\end{center}
see for instance \cite[Theorem 3.1]{Lerman}. Moreover, we have that $\Phi_{K}^{-1}(\Sigma_{\sigma_{K}})  \subset O(\Lambda)$, where $\Sigma_{\sigma_{K}} = {\text{Ad}}(K)\sigma_{K}$, defines an ${\rm{Ad}}$-invariant connected open dense subset in $O(\Lambda)$, see for instance \cite[Proposition 1]{Lane}. From these, we can set $\mathcal{U}_{K} = \Phi_{K}^{-1}(\Sigma_{\sigma_{K}})$. Since $s_{K} \colon \Sigma_{\sigma_{K}} \to \sigma_{K}$ is smooth, by choosing a basis $\{\xi_{1}, \ldots, \xi_{r_{K}}\}$ for $\mathfrak{t}$, where $r_{K} = {\text{rank}}(K)$, and considering the induced dual basis $\{ \phi_{1}, \ldots, \phi_{r_{K}}\} $, it follows that 
\begin{equation}
s_{K} \circ L_{K} = (\phi_{1} \circ s_{K} \circ L_{K})\xi_{1} + \cdots + (\phi_{r_{K}} \circ s_{K} \circ L_{K})\xi_{r_{K}},
\end{equation}
Since $L_{K}$ is equivariant and $s_{K}$ is ${\text{Ad}}$-invariant, it follows from the definition of $\mathcal{U}_{K}$ that $\phi_{i} \circ s_{K} \circ L_{K} \in C^{\infty}(\mathcal{U}_{K})^{{\rm{Ad}}}$, for all $i=1,\dots,r_{K}$. Hence, from the definition of the Hamiltonian flow of $\Phi_{K}^{\ast}(F)$ (see Eq. \ref{flow}) we have that $\phi_{1} \circ s_{K} \circ L_{K}, \ldots, \phi_{r_{K}} \circ s_{K} \circ L_{K}$ are conserved quantities in involution for the Hamiltonian system defined by $(\mathcal{U}_{K},\omega_{O(\Lambda)}|_{\mathcal{U}_{K}},\Phi_{K}^{\ast}(F)|_{\mathcal{U}_{K}})$. Therefore, since $\varrho_{\ast}(\xi_{1}), \ldots, \varrho_{\ast}(\xi_{r_{K}})$ are diagonal matrices and 
\begin{center}
$\det\big (w \mathds{1}_{V} - \varrho_{\ast}(L_{K})\big) = \det\big (w \mathds{1}_{V} - \varrho_{\ast}(s_{K} \circ L_{K})\big),$
\end{center}
on $\mathcal{U}_{K}$, we conclude that the imaginary part of the eigenvalues of $ \varrho_{\ast}(L_{K})$ are conserved quantities in involution for the Hamiltonian system $(\mathcal{U}_{K},\omega_{O(\Lambda)}|_{\mathcal{U}_{K}},\Phi_{K}^{\ast}(F)|_{\mathcal{U}_{K}})$.
\end{proof}

\begin{remark} 
\label{spectrum}
Notice that, in the setting of the previous theorem, fixing some $K$-invariant Hermitian product on $V$, $\forall Z \in O(\Lambda)$ we have that $\varrho_{\ast}(L_{K}(Z))$ is a skew-Hermitian linear operator, so its eigenvalues are all purely imaginary (and possibly zero).
\end{remark}

In what follows, given a compact Lie group $G$, we will denote by $\widehat{G}$ the set of equivalence classes of irreducible (unitary) representations of $G$. Thus, given an irreducible representation $\varrho \colon G \to {\rm{GL}}(V)$, we shall denote by $[\varrho] \in \widehat{G}$ its corresponding class.

\begin{corollary}
\label{C5S5.2Teo5.2.4}
Given a Hamiltonian $G$-space $(O(\Lambda),\omega_{O(\Lambda)},G,\Phi)$, and a chain of closed connected Lie subgroups
$$G = K_{0} \supset K_{1} \supset \cdots \supset K_{s}.$$
If we denote by $\Phi_{i}$ the moment map associated to the Hamiltonian action (by restriction) of each Lie subgroup $K_{i}$ on $O(\Lambda)$, $i = 0,1,\ldots s$, then for every $F \in C^{\infty}(\mathfrak{k}_{s})$, the following hold:

\begin{enumerate}

\item We can associate to the Hamiltonian system $(O(\Lambda),\omega_{O(\Lambda)},\Phi_{s}^{\ast}(F))$ a system of Lax equations
\begin{center}
$\displaystyle \frac{d}{dt}L_{i} + \big [L_{i},P_{i}\big ] = 0$, \ \ \ \ $(i = 1,\ldots,s)$
\end{center}
such that $L_{i} = \Phi_{i}$, and $P_{i} = \nabla (F \circ \pi_{s}^{i})(\Phi_{k})$, where $\pi_{s}^{i} \colon \mathfrak{k}_{i} \to \mathfrak{k}_{s}$ is the projection map;

\item For every $s$-tuple $\varrho = (\varrho_{1},\ldots,\varrho_{s})$, where $[\varrho_{i}] \in \widehat{K_{i}}$, $\forall i = 1,\ldots, s$, there exist a pair of matrix-valued functions $L_{\varrho},P_{\varrho} \colon O(\Lambda) \to \mathfrak{gl}(r,\mathbb{R})$, satisfying the Lax equation
\begin{equation}
\label{laxrep}
\displaystyle \frac{d}{dt}L_{\varrho} + \big [L_{\varrho},P_{\varrho} \big] = 0.
\end{equation}
Moreover, the spectrum $\sigma(L_{\varrho})$ of $L_{\varrho}$ defines a set of conserved quantities in involution for the Hamiltonian system $(O(\Lambda), \omega_{O(\Lambda)},\Phi_{s}^{\ast}(F))$ on an open dense subset $\mathcal{U}_{\Lambda} \subset O(\Lambda)$.
\end{enumerate}
\end{corollary}

\begin{proof}
The first item is a direct consequence of the last theorem, we just need to consider $K_{i} \subset G$ and $F \circ \pi_{s}^{i} \in C^{\infty}(\mathfrak{k}_{i})$, $i = 1,\ldots,s$, for every $F \in C^{\infty}(\mathfrak{k}_{s})$. In this setting, applying the last theorem we get 
$$\displaystyle \frac{d}{dt}L_{i} + \big [L_{i},P_{i} \big] = 0,$$
where $L_{i} = \Phi_{i}$ and $P_{i} = \nabla (F \circ \pi_{s}^{i})(\Phi_{i})$, $\forall i = 1,\ldots,s$. For the second item, given  $\varrho = (\varrho_{1},\ldots,\varrho_{s})$, where $[\varrho_{i}] \in \widehat{K_{i}}$, $\forall i = 1,\ldots, s$, we set
\begin{equation}
L_{\varrho} := \displaystyle \bigoplus_{i=1}^{s}(\varrho_{i})_{\ast}(L_{i}),  \ \ \mbox{  and  }  \ \ P_{\varrho} := \displaystyle \bigoplus_{i=1}^{s}(\varrho_{i})_{\ast}(P_{i}).
\end{equation}
From this, since $(\varrho_{i})_{\ast}$ is a Lie algebra homomorphism, $\forall i = 1,\ldots,s$, we have
\begin{center}
$\Big [L_{\varrho}(\varphi_{t}(Z)), P_{\varrho}(\varphi_{t}(Z)) \Big ] = \displaystyle \bigoplus_{i=1}^{s} \Big [ (\varrho_{i})_{\ast}(L_{i}(\varphi_{t}(Z))),(\varrho_{i})_{\ast}(P_{i}(\varphi_{t}(Z)))\Big ] = \bigoplus_{i=1}^{s}(\varrho_{i})_{\ast}(-\frac{d}{dt}L_{i}(\varphi_{t}(Z)))$.
\end{center}
for all $Z \in O(\Lambda)$, where $\varphi_{t}(Z)$ denotes the Hamiltonian flow of $\Phi_{s}^{\ast}(F)$. Therefore, by linearity of $(\varrho_{i})_{\ast}$, $i = 1,\ldots,s$, we have 
\begin{center}
$\displaystyle \frac{d}{dt}L_{\varrho}(\varphi_{t}(Z)) + \Big [L_{\varrho}(\varphi_{t}(Z)), P_{\varrho}(\varphi_{t}(Z)) \Big ] = 0.$
\end{center}
By taking a suitable $r \in \mathbb{N}$ (large enough), we have $L_{\varrho},P_{\varrho} \colon O(\Lambda) \to \mathfrak{gl}(r,\mathbb{R})$. For the last statement of item 2, we consider $\mathcal{U}_{\Lambda} = \bigcap_{i = 1}^{s}\mathcal{U}_{K_{i}}$, where $\mathcal{U}_{K_{i}} \subset O(\Lambda)$ is obtained from Theorem \ref{prop43}, $\forall i = 1,\ldots,s$. Since the spectrum $\sigma(L_{\varrho})$ of $L_{\varrho}$ is determined by the spectrum of each $(\varrho_{i})_{\ast}(L_{i})$, $i = 1,\ldots,s$, we have from Remark \ref{spectrum}, and  Theorem \ref{prop43}, that $\sigma(L_{\varrho})$ defines a set of conserved quantities in involution for the Hamiltonian system given by $(\mathcal{U}_{\Lambda}, \omega_{O(\Lambda)}|_{\mathcal{U}_{\Lambda}},\Phi_{s}^{\ast}(F)|_{\mathcal{U}_{\Lambda}})$.
\end{proof}
\begin{definition} 
\label{GTPoisson}
Under the hypothesis of Corollary \ref{C5S5.2Teo5.2.4}, we define the {\em Gelfand-Tsetlin commutative Poisson subalgebra} $\Gamma_{O(\Lambda)} \subset C^{\infty}(O(\Lambda))$ as being
\begin{equation}
\Gamma_{O(\Lambda)} := \Big \langle \Phi_{k}^{\ast}\big(\mathcal{S}(\mathfrak{k}_{k})^{{\text{Ad}}} \big) \ \ \Big | \ \ k = 1, \ldots,s \Big \rangle,
\end{equation}
where $\mathcal{S}(\mathfrak{k}_{k})^{{\text{Ad}}}$ denotes the subalgebra of ${\text{Ad}}$-invariant polynomial functions. 
\end{definition}
\begin{remark}
It is worth observing that the motivation for the definition above of Gelfand-Tsetlin Poisson subalgebra is the concept of the Gelfand-Tsetlin subalgebras of the universal enveloping algebras, which are examples of Harish-Chandra subalgebras \cite[p. 87]{FUTORNY2}, \cite{FUTORNY}, see also \cite{Panyushev}.
\end{remark}

\subsubsection{Proof of Theorem 1} From the ideas and results provided in the previous section, we have the following theorem:
\begin{theorem}\label{C5S5.2Teo5.2.10}
Given a Hamiltonian $G$-space $(O(\Lambda),\omega_{O(\Lambda)},G,\Phi)$, where $G$ is either ${\rm{U}}(n) \  \text{or} \ \ {\rm{SO}}(n)$, then there exist $H \in C^{\infty}(O(\Lambda))$, and a pair of matrix-valued functions $L,P \colon (O(\Lambda),\omega_{O(\Lambda)}) \to \mathfrak{gl}(r,\mathbb{R})$, satisfying the Lax equation
\begin{equation}
\displaystyle \frac{d}{dt}L(\varphi_{t}(Z)) + \big [L(\varphi_{t}(Z)), P(\varphi_{t}(Z)) \big ] = 0,
\end{equation}
for all $Z \in O(\Lambda)$, where $\varphi_{t}(Z)$ denotes the flow of the Hamiltonian vector field $X_{H}$ through $Z \in O(\Lambda)$. Moreover, the solutions of the spectral equation 
\begin{equation}
\det\big (w \mathds{1}_{r} - L\big) = 0, 
\end{equation}
provide a maximal set of conserved quantities in involution for the Hamiltonian system $(O(\Lambda),\omega_{O(\Lambda)},H)$ which coincides with the Gelfand-Tsetlin integrable system on $O(\Lambda)$.
\end{theorem}
\begin{proof}

Consider the Hamiltonian $G$-space $(O(\Lambda),\omega_{O(\Lambda)},G,\Phi)$ where $G$ is ${\rm{U}}(n) \  \text{or} \  {\rm{SO}}(n)$. In this setting, we have a chain of subgroups
\begin{center}
${\rm{U}}(n) \supset {\rm{U}}(n-1) \supset \cdots \supset {\rm{U}}(1)$ \ or  \ ${\rm{SO}}(n) \supset {\rm{SO}}(n-1) \supset \cdots \supset {\rm{SO}}(2)$,
\end{center}
where ${\rm{U}}(k)$ is regarded as a Lie subgroup of ${\rm{U}}(n)$ through the identification 
\begin{equation}
\label{blockmatrix}
{\rm{U}}(k) \cong \left (\begin{array}{c|c}
{\rm{U}}(k) \ \ & 0 \\
\midrule
0  & \ \mathds{1}_{n-k}\\
\end{array} \right ) \subset {\rm{U}}(n),
\end{equation}
$\forall k = 1,\ldots,n-1$, and ${\rm{SO}}(k)$ is similarly embedded in the top lefthand corner of ${\rm{SO}}(n)$. From this, by taking $\varrho_{0} = (\varrho_{1},\ldots,\varrho_{s})$, such that $\varrho_{i}$ is either the canonical representation
\begin{center}
$\varrho_{i} \colon {\rm{U}}(n-i) \to {\rm{GL}}(n-i,\mathbb{C})$ \ or \ $\varrho_{i} \colon {\rm{SO}}(n-i) \to {\rm{GL}}(n-i,\mathbb{C}),$
\end{center}
where $i = 1,\ldots,s$, with $s = n-1$ or $s = n-2$, we can apply Corollary \ref{C5S5.2Teo5.2.4}, and obtain $L,P \colon O(\Lambda) \to \mathfrak{u}(r)$, for some suitable $r \in \mathbb{N}$, such that $L = L_{\varrho_{0}}$, and $P = P_{\varrho_{0}}$, satisfying the Lax equation: 
\begin{center}
$\displaystyle \frac{d}{dt}L(\varphi_{t}(Z)) + \big [L(\varphi_{t}(Z)), P(\varphi_{t}(Z)) \big ] = 0,$
\end{center}
for all $Z \in O(\Lambda)$, where $\varphi_{t}(Z)$ denotes the Hamiltonian flow of $H = \Phi_{s}^{\ast}(F) \in C^{\infty}(O(\Lambda))$, for some $F \in C^{\infty}(\sqrt{-1}\mathbb{R})$, here we consider $\mathfrak{u}(1) \cong \sqrt{-1}\mathbb{R} \cong \mathfrak{so}(2)$. Now we observe that, given $Z \in O(\Lambda)$, since 
\begin{center}
$\displaystyle L(Z) = \bigoplus_{i = 1}^{s}L_{i}(Z) = {\text{diag}} \big \{L_{i}(Z) \ \big | \ i = 1,\ldots,s \big\},$
\end{center}
such that $L_{i} = \Phi_{n-i}$, $i = 1,\ldots,s$, with $s = n-1$ or $s = n-2$, it follows that 
\begin{equation}
\label{CharacpolyLax}
\displaystyle \det\big (w \mathds{1}_{r} -  L\big) = \prod_{i = 1}^{s}\det\big (w \mathds{1}_{n-i} - L_{i}\big).
\end{equation}
Therefore, the solutions of the spectral equation associated to $L \colon O(\Lambda) \to \mathfrak{u}(r)$ is determined by the solutions of the spectral equation associated to each $L_{i}$, $i = 1,\ldots,s$, with $s = n-1$ or $s = n-2$. From this last fact, we have that the spectrum $\sigma(L,\Lambda)$ of $L$ generates the Gelfand-Tsetlin commutative Poisson subalgebra $\Gamma_{O(\Lambda)}$, see Definition \ref{GTPoisson}. Hence, since the action of ${\rm{U}}(k)$ on adjoint orbits of ${\rm{U}}(k+1)$ is coisotropic, and the same holds for the action of  ${\rm{SO}}(k)$ on adjoint orbits of ${\rm{SO}}(k+1)$, see for instance \cite{Panyushev}, \cite{INTMULT}, we have that $\sigma(L,\Lambda)$ defines an integrable system on $O(\Lambda)$, which turns out to be the Gelfand-Tsetlin integrable system. 
\end{proof}

\begin{remark}
\label{opendensesetsubemrsion}
In the construction provided in the last theorem, the maximal subset of algebraically independent Poisson-commuting functions of $\sigma(L,\Lambda)$ can be obtained by looking at the Gelfand-Tsetlin patterns defined by $\Lambda$. Also, it is worth observing that the integrable system provided by $\sigma(L,\Lambda)$ defines a map $\mathscr{H} \colon (O(\Lambda),\omega_{O(\Lambda)}) \to \mathbb{R}^{\frac{\dim_{\mathbb{R}}(O(\Lambda))}{2}}$, which is a smooth submersion on an open dense subset $O(\Lambda)^{\mathscr{H}} := \mathcal{U}_{\Lambda} \subset O(\Lambda)$, see item 2 of Corollary \ref{C5S5.2Teo5.2.4}, and continuous on the whole orbit $O(\Lambda)$. In fact, the components of $\mathscr{H}$ are defined by Guillemin and Sternberg’s action coordinates (Eq. \ref{GSAction}).
\end{remark}

\begin{remark}
\label{RemarkGiacobbe}
The result of Theorem \ref{C5S5.2Teo5.2.10} generalizes the constriction introduced in \cite[Proposition 1.4]{Giacobbe} in the following sense: In the particular setting given by $(O(\Lambda),\omega_{O(\Lambda)},{\rm{U}}(n),\Phi)$, such that $O(\Lambda) \subset \mathfrak{u}(n)$ is some regular orbit, from the last theorem one can consider the Hamiltonian system 
\begin{center}
$\big (O(\Lambda),\omega_{O(\Lambda)},H = \pm (\sqrt{-1})^{k}{\text{Tr}}(L^{k})\big)$.
\end{center}
It follows trivially that $\{\sigma(L,\Lambda),H\}_{O(\Lambda)} = 0$, i.e., the Lax pair formulation of the last theorem gives rise to the $\frac{n(n-1)}{2}$
integrals of motion for the Hamiltonian system above. For the particular case that $k = 2$, the Hamiltonian system defined by $H = -{\text{Tr}}(L^{2})$ coincides with the Hamiltonian system studied in \cite[Proposition 1.4]{Giacobbe}. In order to recover the result \cite[Theorem 1.8]{Giacobbe} from Theorem \ref{C5S5.2Teo5.2.10}, we need to take a {\textit{parametric spectral deformation}} on $L$ given by
\begin{equation}
L(\zeta;Z) := \zeta \mathds{1}_{r} + L(Z), \ \ (Z \in O(\Lambda))
\end{equation}
for some complex parameter $\zeta \in \mathbb{C}$. By observing that the coefficients of the characteristic polynomial of $L(\zeta;Z)$ are determined by the coefficients of the characteristic polynomial of $L(Z)$, for all $Z \in O(\Lambda)$, see for instance \cite[Lemma 1.5]{Giacobbe}, we recover the desired result.
\end{remark}

\section{Gelfand-Tsetlin integrable systems and algebraic curves}
\label{GThyperelliptic}
The aim of this section is to prove Theorem \ref{Theo2}. In order to do this, we start by reviewing some results of \cite{Cho} concerning with the topology of the fibers of Gelfand-Tsetlin integrable systems on ${\rm{U}}(n)$-adjoint orbits. After that, following \cite{Kirwan}, \cite{Schlag}, \cite{Girondo}, \cite{Knapp}, \cite{Bosch}, \cite{Fischer}, we shall review some basic generalities on Abelian integrals, algebraic curves and Riemann surfaces. 

\subsection{Gelfand-Tsetlin integrable systems on ${\rm{U}}(n)$-adjoint orbits}
\label{GTU(n)}

From now on, for all $n \in \mathbb{N}$, we shall consider the standard maximal torus $T^{n} \subset {\rm{U}}(n)$, with Lie algebra $\mathfrak{t}_{n} \subset \mathfrak{u}(n)$ defined by the diagonal matrices. Moreover, in order to avoid problems with signs related to imaginary eigenvalues, we will identify $X \in \mathfrak{u}(n)$ with $-\sqrt{-1}X \in \sqrt{-1}\mathfrak{u}(n)$. Since under this last identification we have $\mathfrak{t}_{n} \cong \mathbb{R}^{n}$, we will consider the following realization for the closed positive Weyl chamber: 
\begin{center}
$(\mathfrak{t}_{n})_{+} = \Big \{ {\text{diag}}\big\{\lambda_{1},\ldots,\lambda_{n}\} \in \mathfrak{t}_{n} \ \Big | \ \lambda_{1} \geq \lambda_{2} \geq \cdots \geq \lambda_{n} \Big\}$.
\end{center}
Also, we shall fix the nested chain of subgroups given by
\begin{center}
${\rm{U}}(n) \supset {\rm{U}}(n-1) \supset \cdots \supset {\rm{U}}(1)$,
\end{center}
where ${\rm{U}}(k)$ is regarded as subgroup of ${\rm{U}}(n)$ as in Eq. \ref{blockmatrix}, for every $k = 1,\ldots,n-1$. We also shall consider the following data: In the setting of Theorem \ref{C5S5.2Teo5.2.10}, given an adjoint orbit $O(\Lambda) \subset \mathfrak{u}(n)$, we have $L \colon O(\Lambda) \to \mathfrak{u}(r)$, where $r = \frac{n(n-1)}{2}$, such that 
\begin{center}
$\displaystyle L(Z) = {\text{diag}} \big \{L_{i}(Z) \ \big | \ i = 1,\ldots,n-1 \big\},$
\end{center}
for all $Z \in O(\Lambda)$. Since $L_{i} \colon O(\Lambda) \to \mathfrak{u}(i)$, $i = 1,\ldots,n-1$, we have
\begin{equation}
\det\big (w \mathds{1}_{i} -  L_{i}(Z)\big) = \big (w - \lambda_{1}^{(i)}(Z)\big )  \cdots \big (w - \lambda_{i}^{(i)}(Z) \big ),
\end{equation}
for all $Z \in O(\Lambda)$, and for all $i = 1,\ldots,n-1$. From these, we have the spectrum $\sigma(L,\Lambda)$ of $L$ given by the functions
\begin{equation}
\label{Laxspectrum}
\sigma(L,\Lambda) = \big \{ \lambda_{j}^{(k)} \ \big | \ 1 \leq j \leq k, \  1 \leq k \leq n-1\big \}.
\end{equation}
Notice that from {\textit{Cauchy’s interlace theorem}} \cite{Hwang}, we have that
\begin{equation}
\lambda_{1}^{(k)} \geq \lambda_{1}^{(k-1)} \geq \lambda_{2}^{(k)} \geq \cdots \geq \lambda_{k-1}^{(k)} \geq \lambda_{k-1}^{(k-1)} \geq \lambda_{k}^{(k)},
\end{equation}
for all $ 1 \leq k \leq n-1$. The inequalities above define the Gelfand-Tsetlin patterns (Eq. \ref{GTpattern}). 
\begin{equation}
\label{GTpattern}
\begin{array}{*{30}c} 
\lambda_{1}&&&&\lambda_{2}&&&&\lambda_{3}&&\cdots&&\lambda_{n-1}&&&&\lambda_{n}\\  
&\rotatebox[origin=c]{-45}{$\geq$}&&\rotatebox[origin=c]{50}{$\geq$}&&\rotatebox[origin=c]{-45}{$\geq$}&&\rotatebox[origin=c]{50}{$\geq$}&&&&&&\rotatebox[origin=c]{-45}{$\geq$}&&\rotatebox[origin=c]{50}{$\geq$}\\ 
&&\lambda_{1}^{(n-1)}&&&&\lambda_{2}^{(n-1)}&&&&&&&&\lambda_{n-1}^{(n-1)}&&\\
&&&\rotatebox[origin=c]{-45}{$\geq$}&&\rotatebox[origin=c]{50}{$\geq$}&&&&&&&&\rotatebox[origin=c]{50}{$\geq$}&&\\
&&&&\lambda_{1}^{(n-2)}&&&&&&&&\lambda_{n-2}^{(n-2)}&&&&\\
&&&&&\rotatebox[origin=c]{-45}{$\geq$}&&&&&&\rotatebox[origin=c]{45}{$\geq$}&&&&&\\
&&&&&&\ddots&&&&\rotatebox[origin=c]{85}{$\ddots$}&&\\
&&&&&&&\rotatebox[origin=c]{-45}{$\geq$}&&\rotatebox[origin=c]{50}{$\geq$}&&&\\
&&&&&&&&\lambda_{1}^{(1)}&&&&
\end{array}
\end{equation}
From above we have the following definition.
\begin{definition}
\label{defGTsystem}
Given $\Lambda = \text{diag}\{\lambda_{1}, \ldots,\lambda_{n}\}$, such that $\lambda_{1} \geq \lambda_{2} \geq \cdots \geq \lambda_{n}$, the Gelfand-Tsetlin polytope $\Delta_{\Lambda}$ is defined by the collection of points ${\bf{u}} = ({\bf{u}}_{j}^{(k)})$ satisfying the Gelfand-Tsetlin patterns.
\end{definition}
Consider now the following index set 
\begin{equation}
\label{indexintsystem}
\mathcal{I}(\Lambda) = \big \{(j,k) \ \big | \ j,k \in \mathbb{Z}_{\geq 1}, \ \lambda_{j}^{(k)} \ {\text{is not a constant function on}} \ O(\Lambda) \big \},
\end{equation}
so that $|\mathcal{I}(\Lambda)| = \frac{1}{2}\dim_{\mathbb{R}}(O(\Lambda))$. Then we have the following results.

\begin{proposition}[\cite{GELFORB}]
The Gelfand-Tsetlin polytope $\Delta_{\Lambda}$ coincides with the image of the map 
\begin{equation}
\mathscr{H} \colon O(\Lambda) \to \mathbb{R}^{|\mathcal{I}(\Lambda)|}, \ \ \ \mathscr{H} = (\lambda_{j}^{(k)})_{(j,k) \in \mathcal{I}(\Lambda)},
\end{equation}
that is, $\Delta_{\Lambda} = \mathscr{H}(O(\Lambda))$.
\end{proposition}
\begin{proposition}[Proposition 5.3 in \cite{GELFORB}]
For each $(j,k) \in \mathcal{I}(\Lambda)$, we have that $\lambda_{j}^{(k)} \colon O(\Lambda) \to \mathbb{R}$ is continuous in every point of $O(\Lambda)$, and smooth at $Z \in O(\Lambda)$ if
\begin{equation}
\label{smoothinequality}
\lambda_{j}^{(k+1)}(Z) > \lambda_{j}^{(k)}(Z) > \lambda_{j+1}^{(k+1)}(Z).
\end{equation}
In particular, $\mathscr{H} = (\lambda_{j}^{(k)})_{(j,k) \in \mathcal{I}(\Lambda)}$ is smooth on the open dense subset $O(\Lambda)^{\mathscr{H}} = \mathscr{H}^{-1}({\rm{int}}(\Delta_{\Lambda}))$. Furthermore, the functions $\{\lambda_{j}^{(k)}\}_{(j,k) \in\mathcal{I}(\Lambda) }$ are functionally independent precisely in $O(\Lambda)^{\mathscr{H}}$.
\end{proposition}
Let $\mathcal{F}$ be a $d$-dimensional face of the Gelfand-Tsetlin polytope $\Delta_{\Lambda}$. Denoting by ${\rm{int}}(\mathcal{F})$ its relative interior, we have the following result.

\begin{lemma}[Lemma 6.7 in \cite{Cho}]
\label{actionface}
For each $d$-dimensional face $\mathcal{F}$ of $\Delta_{\Lambda}$, there exists $\mathcal{I}(\mathcal{F}) \subset \mathcal{I}(\Lambda)$, such that $|\mathcal{I}(\mathcal{F})| = d$, and for every $(j,k) \in \mathcal{I}(\mathcal{F})$, the component $\lambda_{j}^{(k)}$ is smooth on $\mathscr{H}^{-1}({\rm{int}}(\mathcal{F}))$. Furthermore, for every ${\bf{u}} \in {\rm{int}}(\mathcal{F})$, we have a fiberwise $T^{d}$-action on $\mathscr{H}^{-1}({\bf{u}})$ generated by $\{\lambda_{j}^{(k)}\}_{(j,k) \in \mathcal{I}(\mathcal{F})}.$
\end{lemma}

\begin{theorem}[\cite{Cho}] 
\label{fiberdescription}
Let ${\bf{u}}$ be a point lying on the relative interior of a $d$-dimensional face of the Gelfand-Tsetlin polytope $\Delta_{\Lambda}$. Then there is a $T^{d}$-equivariant diffeomorphism 
\begin{equation}
\Psi \colon \mathscr{H}^{-1}({\bf{u}}) \to T^{d} \times Y({\bf{u}}),
\end{equation}
where 
\begin{itemize}
\item[1)] a free $T^{d}$-action is given by Lemma \ref{actionface};
\item[2)] $T^{d}$ acts freely on the left factor $T^{d} \times Y({\bf{u}})$, and $Y({\bf{u}}) := \mathscr{H}^{-1}({\bf{u}})/T^{d}$.
\end{itemize}
\end{theorem}

\begin{remark}
\label{Topfibers}
It is worth pointing out that in the above theorem we also have the following facts:
\begin{enumerate}
\item[(i)] $\mathscr{H}^{-1}({\bf{u}}) = E_{n-1} \to E_{n-2} \to \cdots \to E_{1} \to E_{0} = {\text{point}}$ (iterated bundle structure);
\item [(ii)]$E_{k} \to E_{k - 1}$ is a $S_{k}$-bundle, where $S_{k}$ is a product of odd-dimensional spheres. Moreover, we have
\begin{equation}
\displaystyle{\dim_{\mathbb{R}}(\mathscr{H}^{-1}({\bf{u}})) = \sum_{k = 1}^{n-1}\dim_{\mathbb{R}}(S_{k})};
\end{equation}
\item[(iii)] $T^{d} = T^{d_{1}} \times T^{d_{2}} \times \cdots \times T^{d_{n-1}}$, such that $d_{k} = $ ``number of $S^{1}$-factors occurring in $S_{k}"$;
\item [(iv)]$Y({\bf{u}}) = Y_{n-1}({\bf{u}}) \to Y_{n-2}({\bf{u}}) \to \cdots \to Y_{1}({\bf{u}}) \to Y_{0}({\bf{u}}) = $ point (iterated bundle structure), such that 
\begin{equation}
Y_{k}({\bf{u}}) = E_{k}/T^{d_{1} + \cdots +d_{k}}.
\end{equation}
Moreover, $Y_{k}({\bf{u}}) \to Y_{k-1}({\bf{u}})$ is a $Y_{k}$-bundle, such that $Y_{k} = S_{k}/T^{d_{1} + \cdots +d_{k}}.$
\end{enumerate}
The results listed above provide a very rich description for the topology of the fibers of the Gelfand-Tsetlin integrable system, more details can be found in \cite[Theorem 4.12 and Theorem 6.11]{Cho}. 
\end{remark}

As a consequence of the previous facts, we have the following result.

\begin{corollary}[\cite{Cho}]
\label{lagranguantorusfiber}
For every ${\bf{u}} \in \Delta_{\Lambda}$, we have that $\mathscr{H}^{-1}({\bf{u}})$ is a Lagrangian torus if and only if ${\bf{u}} \in {\rm{int}}(\Delta_{\Lambda})$.
\end{corollary}

\subsection{Generalities on Abelian integrals and algebraic curves}

Given a Riemann surface $M$ and an open set $U \subseteq M$, we will denote by $\mathscr{O}(U)$ and by $\mathscr{M}(U)$ the ring of holomorphic functions and the field of meromorphic functions on $U$, respectively. Also, we will denote by $\mathscr{M}^{1}(U)$ the module of Abelian differentials (meromorphic 1-forms) defined on $U \subseteq M$. 
\begin{definition} An Abelian integral is an integral of the form
\label{defabel}
\begin{equation}
\label{abelint}
\int_{\gamma}R(w,\zeta(w))dw,
\end{equation}
where $\gamma \colon [0,1] \to \mathbb{C}$ is a continuous path in $\mathbb{C}$, $R(w,\zeta)$ is a rational function of two variables $w$ and $\zeta$, and $\zeta(w)$ is a continuous function of $w$ defined on the image of the path $\gamma$ such that $\zeta(w)$ and $w$ satisfy a polynomial relation $h(w,\zeta(w)) = 0$, for some $h(w,\zeta) \in \mathbb{C}[w,\zeta]$.
\end{definition}

In the setting of Definition \ref{defabel} we are interested in the particular case where $h(w,\zeta) = \zeta^{2} - P(w)$, for some $P(w) \in \mathbb{C}[w]$, such that $\deg(P) = N$, for some $N \in \mathbb{Z}_{\geq 0}$. In order to deal with this case from the modern perspective of Riemann surfaces and algebraic curves, let us collect some general facts about affine curves\footnote{For us an affine (algebraic) curve is not necessary irreducible.} of the form $\mathcal{C}:  \zeta^{2} - P(w) = 0$. In the general setting, we have the following facts about such a curve $\mathcal{C}$:
\begin{enumerate}
\item $\mathcal{C}$ can be regarded as the affine scheme ${\rm{Spec}}\Big(\frac{\mathbb{C}[w,\zeta]}{\langle \zeta^{2} - P(w) \rangle}\Big )$;
\item $\mathcal{C}$ is a smooth curve if and only if $P(w) = (w-\alpha_{1}) \cdots (w-\alpha_{N})$, such that $\alpha_{i} \neq \alpha_{j}$, $\forall i \neq j$. In this case, the algebraic curve $\mathcal{C}$ is called hyperelliptic curve and the points $p_{j} = (\alpha_{j},0) \in \mathcal{C}$, $j = 1,\ldots,N$, are called branch points of $\mathcal{C}$;
\item If $\mathcal{C}$ is singular, then its set of singular points ${\rm{Sing}}(\mathcal{C})$ is given by
\begin{equation}
{\rm{Sing}}(\mathcal{C}) = \big \{ (\alpha,0) \in \mathbb{C}^{2} \ \big | \ P(\alpha) = 0, \ {\rm{mult}}_{\alpha}(P) > 1\big \};
\end{equation}
\item $\mathcal{C}$ is irreducible if and only if $P(w)$ is not a square of an element of $\mathbb{C}[w]$;
\item If $P(w) = (q(w))^{2}$, for some $q(w) \in \mathbb{C}[w]$, then $\mathcal{C}$ decomposes into two irreducible smooth components $\mathcal{C}_{i}$, i = 1,2, such that 
\begin{center}
$\mathcal{C}_{1}: \zeta - q(w) = 0$ \ \ and \ \ $\mathcal{C}_{2}: \zeta + q(w) = 0$;
\end{center}
\item $\mathcal{C}$ is irreducible if and only if its set of regular points $\mathcal{C}_{\rm{reg}} := \mathcal{C} \backslash {\rm{Sing}}(\mathcal{C})$ is connected. In this last case $\mathcal{C}_{\rm{reg}}$ defines a non-compact Riemann surface;
\item If $\mathcal{C}$ is irreducible and $\Sigma$ is the compact Riemann surface defined by the compactification of $\mathcal{C}$, then we have that $w,\zeta \in \mathscr{M}(\Sigma)$ and
\begin{equation}
\mathscr{M}(\Sigma) \cong \mathbb{C}(w)[\zeta]/\langle \zeta^{2} - P(w) \rangle,
\end{equation}
in the sense of fields isomorphisms. Here $\langle \zeta^{2} - P(w) \rangle$ is the ideal generated by $h(w,\zeta)$ in $\mathbb{C}(w)[\zeta]$.
\item For every $f \in \mathscr{M}(\Sigma)$ there exists a rational function $R(w,\zeta)$, such that $f = R(w,\zeta)$. Conversely, every rational function $R(w,\zeta)$ whose denominator is not divisible by $h(w,\zeta) = \zeta^{2} - P(w)$ gives rise to a meromorphic function on $\Sigma$.
\end{enumerate}

Let $\mathcal{C}$ be an algebraic curve as above and suppose that $\mathcal{C}$ is irreducible. By considering $\mathcal{C}_{\rm{reg}}$ as a non-compact Riemann surface, we have $\mathcal{C}_{\rm{reg}} = \Sigma \backslash S$, where $S \subset \Sigma$ is a finite set. For any rational function $R(w,\zeta)$ whose denominator is not divisible by $h(w,\zeta) = \zeta^{2} - P(w)$, we have an associated abelian differential $\Omega := R(w,\zeta)dw$ on $\mathcal{C}_{\rm{reg}}$. From this, given a piecewise-smooth path $\gamma \colon [0,1] \to \mathbb{C}$, such that $P(\gamma(t)) \neq 0$, $\forall t \in [0,1]$, and a continuous function $\zeta(w)$ of $w$ defined on the image of the path $\gamma$, satisfying $\zeta(w)^{2} = P(w)$, we have 
\begin{equation}
  \int_{c}\Omega = \int_{\gamma}R(w,\zeta(w))dw, 
\end{equation}
where $c$ is the piecewise-smooth\footnote{Since $\zeta(w)^{2} = P(w)$ and $P(\gamma(t)) \neq 0$, $\forall t \in [0,1]$, we have $\frac{d}{dt}\zeta(\gamma(t)) = \frac{1}{2\zeta(\gamma(t))}\frac{d}{dt}P(\gamma(t))$, $\forall t \in [0,1]$ where $\gamma$ is smooth.} path obtained from the composition of $\gamma$ with the map $w \mapsto (w,\zeta(w))$. From the ideas above, we have the following definition.

\begin{definition}
A hyperelliptic integral is an integral of the form
\begin{equation}
\int_{c}R(w,\zeta)dw,
\end{equation}
where $R(w,\zeta)$ is a rational function on a hyperelliptic curve $\mathcal{C}:  \zeta^{2} - P(w) = 0,$ and $c \colon I \to \mathcal{C}$ is a piecewise-smooth path not passing through any pole of the meromorphic differential $R(w,\zeta)dw$.
\end{definition}
Let $\mathcal{C}:  \zeta^{2} - P(w) = 0$ be a hyperelliptic curve. Given a rational function $R(w,\zeta)$ on $\mathcal{C}$, by using the algebraic equation which defines $\mathcal{C}$, we have that $R(w,\zeta)$ can be written in the following form\footnote{See for instance \cite[p. 259]{Schlag}.} 
\begin{equation}
R(w,\zeta) = \frac{A(w) + \zeta B(w)}{C(w)},
\end{equation}
such that $A(w),B(w),C(w) \in \mathbb{C}[w]$. Conversely, any function as above gives rise to a rational function on $\mathcal{C}$. This last fact enables us to explicitly construct meromorphic differentials on $\mathcal{C}$ with certain prescribed residues. The next example will play an important role in the proof of Theorem \ref{Theo2}.
\begin{example}
\label{prescribedpoles}
Given a hyperelliptic curve $\mathcal{C}:  \zeta^{2} - P(w) = 0$, such that $P(w) = (w-\alpha_{1}) \cdots (w-\alpha_{N})$, we can consider the following meromorphic differential on $\mathcal{C}$
\begin{equation}
\Omega := \frac{1}{2}\frac{wdw}{w -\alpha_{k}}.
\end{equation}
for some $1 \leq k \leq N$. By definition $\Omega$ has one pole on $\mathcal{C}$ which coincides with the branch point $p_{k} = (\alpha_{k},0) \in \mathcal{C}$. By taking $\epsilon > 0$ small enough, such that $P(w) \neq 0$ for all $w \in \overline{B(\alpha_{j};\epsilon^{2})}\backslash \{\alpha_{j}\}$, for every $j = 1,\ldots,N$, we can define a smooth path $c_{j} \colon [0,2\pi] \to \mathcal{C}$, such that\footnote{The curve $c_{j}$ is just the composition of the loop $t \mapsto {\rm{e}}^{t\sqrt{-1}}$ with the standard local chart at $p_{j} = (\alpha_{j},0) \in \mathcal{C}$, see \cite[p. 141]{Schlag}.} 
\begin{equation}
c_{j}(t) := \bigg ( \epsilon^{2} {\rm{e}}^{2t \sqrt{-1}} + \alpha_{j}, \epsilon{\rm{e}}^{t \sqrt{-1}}\sqrt{\prod_{i \neq j, i = 1 }^{N}(\epsilon^{2}{\rm{e}}^{2t \sqrt{-1}} + \alpha_{j} - \alpha_{i})}\bigg).
\end{equation}
By definition we have that $c_{j}$ is a small loop around $p_{j} = (\alpha_{j},0) \in \mathcal{C}$, for every $j = 1,\ldots,N$, and from a straightforward application of the residue theorem it follows that
\begin{equation}
 \frac{1}{2\pi\sqrt{-1}}\oint_{c_{j}}\Omega =  \frac{1}{4\pi \sqrt{-1}}\oint_{ \gamma_{j}}\frac{wdw}{w -\alpha_{k}} = \frac{1}{2}{\rm{Ind}}(\gamma_{j};\alpha_{j}){\rm{Res}}\bigg (\frac{w}{w -\alpha_{k}};\alpha_{j} \bigg ),
\end{equation}
where $\gamma_{j} \colon [0,2\pi] \to \mathbb{C}$ is the closed curve $\gamma(t) =  \epsilon^{2} {\rm{e}}^{2t \sqrt{-1}} + \alpha_{j}$. Since 
\begin{center}
$\displaystyle {\rm{Ind}}(\gamma_{j};\alpha_{j}) = \frac{1}{2\pi\sqrt{-1}}\oint_{\gamma_{j}}\frac{dw}{w - \alpha_{j}} = 2$ \ \ and \ \ ${\rm{Res}}\big (\frac{w}{w -\alpha_{k}};\alpha_{j} \big ) = \alpha_{j}{\rm{Res}}\big (\frac{1}{w -\alpha_{k}};\alpha_{j} \big ) = \begin{cases}\alpha_{k}, \ {\text{if}} \ j = k \\ 0, \ {\text{if}}  \ j \neq k \end{cases}$,
\end{center}
we conclude that ${\rm{Res}}\big (\Omega;p_{k} \big ) = \alpha_{k}$ and ${\rm{Res}}\big (\Omega;p_{j} \big ) = 0$, $\forall j \neq k$.
\end{example}
\subsubsection{Proof of Theorem 2}

In order to prove Theorem \ref{Theo2}, let us give the general recipe of how to assign a family of algebraic curves to each adjoint orbit of ${\rm{U}}(n)$. In the general setting, let $\Lambda \in \mathfrak{u}(n) \backslash \{0\}$ be an arbitrary diagonal matrix with distinct eigenvalues $\lambda_{n_{1}},\ldots,\lambda_{n_{s}} = \lambda_{n}$, $s>0$, i.e., $\Lambda = \text{diag}\{\lambda_{1}, \ldots,\lambda_{n}\}$, where
\begin{equation}
\label{eigenmult}
\underbracket{\lambda_{1} = \cdots = \lambda_{n_{1}}}_{k_{1}} >\underbracket{\lambda_{n_{1} + 1} = \cdots = \lambda_{n_{2}}}_{k_{2}} > \cdots > \underbracket{\lambda_{n_{s-1}+1} = \cdots = \lambda_{n}}_{k_{s}},
\end{equation}
such that $k_{1} +  \cdots + k_{s} = n$. In this case, we have 
\begin{equation}
O(\Lambda) \cong {\rm{U}}(n)/({\rm{U}}(k_{1}) \times \cdots \times {\rm{U}}(k_{s})),
\end{equation}
such that $\dim_{\mathbb{R}}(O(\Lambda)) = n^{2} - \sum_{i}k_{i}^{2}$. If we consider the associated Lax matrix $L \colon O(\Lambda) \to \mathfrak{gl}(r,\mathbb{R})$, given by Theorem \ref{C5S5.2Teo5.2.10}, we have the following result.

\begin{lemma} 
\label{divisor}
In the above setting, for every $Z \in O(\Lambda)$, we have
\begin{equation}
\label{characteristiceint}
\det( w\mathds{1}_{r} - L(Z)) = (w - \lambda_{n_{1}})^{m_{1}} \cdots (w-\lambda_{n_{s}})^{m_{s}}F_{\Lambda}(w;Z), 
\end{equation}
where $m_{j} = \frac{(k_{j}-1)k_{j}}{2}$, $\forall j = 1,\ldots,s$, and $F_{\Lambda}(w;Z) \in \mathbb{C}[w]$.
\end{lemma}

\begin{proof}
Given $\Lambda = \text{diag}\{\lambda_{1}, \ldots,\lambda_{n}\}$, suppose that for some $\ell \in \mathbb{N}$ we have
\begin{equation}
\lambda_{1} \geq \cdots \geq \lambda_{i} > \underbrace{\lambda_{i+1} = \lambda_{i+2} = \cdots = \lambda_{i+\ell - 1} = \lambda_{i + \ell}}_{\ell} > \lambda_{i+\ell + 1} \geq \cdots \geq \lambda_{n}.
\end{equation}
Denoting $\lambda_{i+1} = \cdots = \lambda_{i + \ell} = c_{0}$, since $\lambda_{j} \geq \lambda_{j}^{(n-1)} \geq \lambda_{j+1}$, for all $1 \leq j \leq n-1$, it follows that $\lambda_{j}^{(n-1)} = c_{0}$, for all $i+1 \leq j \leq i+\ell-1$, that is
\begin{equation}
\cdots \geq \lambda_{i}^{(n-1)} \geq \underbrace{\lambda_{i+1}^{(n-1)} = \lambda_{i+2}^{(n-1)} = \cdots = \lambda_{i+\ell - 1}^{(n-1)}}_{\ell-1}\geq \lambda_{i+\ell}^{(n-1)} \geq \cdots.
\end{equation}
By observing that $\lambda_{j}^{(n-1)} \geq \lambda_{j}^{(n-2)} \geq \lambda_{j+1}^{(n-1)}$, for all $1 \leq j \leq n-2$, from a similar argument as above, we have that $\lambda_{j}^{(n-2)} = c_{0}$, for all $ i+1 \leq j \leq i+\ell-2$. By applying the above argument inductively, we conclude that 
\begin{equation}
\underbrace{\lambda_{i+1}^{(n-k)} = \lambda_{i+2}^{(n-k)} = \cdots = \lambda_{i+\ell - k}^{(n-k)}}_{\ell-k} = c_{0},
\end{equation}
for all $1 \leq k \leq \ell-1$. Hence, we have that $(w - c_{0})^{m}$, such that $m = \frac{\ell(\ell-1)}{2}$, is a divisor of the polynomial  $\det( w\mathds{1}_{r} - L(Z)) \in \mathbb{C}[w]$, for all $Z \in O(\Lambda)$. From this, if Eq. \ref{eigenmult} holds, we have that 
\begin{equation}
\det( w\mathds{1}_{r} - L(Z)) = (w - \lambda_{n_{1}})^{m_{1}} \cdots (w-\lambda_{n_{s}})^{m_{s}}F_{\Lambda}(w;Z), 
\end{equation}
for all $Z \in O(\Lambda)$, where $m_{j} = \frac{(k_{j}-1)k_{j}}{2}$, $\forall j = 1,\ldots,s$, and $F_{\Lambda}(w;Z) \in \mathbb{C}[w]$.
\end{proof}

In the setting of the previous lemma, since $k_{1} +  \cdots + k_{s} = n$ and $r = \frac{n(n-1)}{2}$, it follows that
\begin{equation}
\deg(F_{\Lambda}(w;Z)) = \frac{n(n-1)}{2} - \sum_{j = 1}^{s}\frac{(k_{j}-1)k_{j}}{2} = \frac{n^{2} - \sum_{i}k_{i}^{2}}{2} = \frac{\dim_{\mathbb{R}}(O(\Lambda))}{2}. 
\end{equation}
Actually, from the Gelfand-Tsetlin pattern (Eq. \ref{GTpattern}) and Definition \ref{defGTsystem}, we have that
\begin{equation}
\label{branchintegrable}
F_{\Lambda}(w;Z) = \prod_{(j,k) \in \mathcal{I}(\Lambda)}\big (w - \lambda_{j}^{(k)}(Z)\big),
\end{equation}
for all $Z \in O(\Lambda)$. By keeping the previous notation, we have the following result.
\begin{lemma}
\label{fundequation}
There exists a rational function $r_{\Lambda}(w) \in \mathbb{C}(w)$, such that, for all $Z \in O(\Lambda)$, the following equation holds
\begin{equation} 
\label{minimalspectral}
\textstyle{r_{\Lambda}(w)\det( w\mathds{1}_{r} - L(Z)) = \mu_{\Lambda}(w) \Big [\prod_{(j,k) \in \mathcal{I}(\Lambda)}\big (w - \lambda_{j}^{(k)}(Z)\big) \Big ]},
\end{equation}
where $\mu_{\Lambda}(w) = (w - \lambda_{n_{1}}) \cdots (w-\lambda_{n_{s}}) \in \mathbb{C}[w]$ is the minimal polynomial of $\Lambda \in \mathfrak{u}(n)$.
\end{lemma}
\begin{proof}
As before, let $\Lambda \in \mathfrak{u}(n) \backslash \{0\}$ be an arbitrary diagonal matrix with distinct eigenvalues $\lambda_{n_{1}},\ldots,\lambda_{n_{s}}$, with multiplicities given respectively by $k_{1}, \ldots, k_{s}$, such that $s > 0$. By setting
\begin{equation}
r_{\Lambda}(w):= \frac{1}{(w - \lambda_{n_{1}})^{m_{1}-1} \cdots (w-\lambda_{n_{s}})^{m_{s}-1}},
\end{equation}
such that $m_{j} = \frac{(k_{j}-1)k_{j}}{2}$, $\forall j = 1,\ldots,s$, it follows from Lemma \ref{divisor} that Eq. \ref{minimalspectral} holds for all $Z \in O(\Lambda)$.
\end{proof}
From above, for every $Z \in O(\Lambda)$ one can define a complex algebraic curve $\mathcal{C}_{Z}$ by setting
\begin{equation}
\label{zerolocuscharacteristic}
\mathcal{C}_{Z} :  \zeta^{2} - P_{\Lambda}(w;Z) = 0,
\end{equation}
where $P_{\Lambda}(w;Z) = r_{\Lambda}(w)\det( w\mathds{1}_{r} - L(Z))$. The motivation behind Eq. \ref{zerolocuscharacteristic} is the following. In the setting of Remark \ref{Topfibers}, let ${\bf{u}} \in \Delta_{\Lambda}$ be a point lying on the relative interior of a $d$-dimensional face. Considering the iterated bundle
\begin{equation}
\label{itbundle}
\mathscr{H}^{-1}({\bf{u}}) = E_{n-1} \to E_{n-2} \to \cdots \to E_{1} \to E_{0} = {\text{point}},
\end{equation}
we have that the description of $\mathscr{H}^{-1}({\bf{u}})$ provided in Theorem \ref{actionface} depends on each $S_{k}$-bundle $E_{k} \to E_{k - 1}$. By choosing any $Z \in \mathscr{H}^{-1}({\bf{u}})$, for each $k = 1,\ldots,n-1$, we have that $S_{k}$ is identified with the set of points $(\xi_{1},\ldots,\xi_{k}) \in \mathbb{C}^{k}$, satisfying either
\begin{equation}
\begin{cases} \det\big ( w\mathds{1}_{n} - \Lambda \big ) = \det\big( w\mathds{1}_{n-1} - L_{n-1}(Z)\big)\Bigg [ (w - c_{n-1}) - \displaystyle{\sum_{i = 1}^{n-1}}\frac{|\xi_{i}|^{2}}{w - \lambda_{i}^{(n-1)}(Z)} \Bigg ]
, \ \ if \ \ k = n-1,\\
\ \ \ \ \ \ \ \ \ \ \ \ \ \ \ \ \ \ \ \ \ \ \ \ \ \ \ \ \ \ \ \ \ \ {\text{or}} \\ 
\det\big ( w\mathds{1}_{k+1} - L_{k+1}(Z)\big ) = \det\big( w\mathds{1}_{k} - L_{k}(Z)\big)\Bigg [ (w - c_{k}) - \displaystyle {\sum_{i = 1}^{k}}\frac{|\xi_{i}|^{2}}{w - \lambda_{i}^{(k)}(Z)} \Bigg ], \ \ if \ \ 1 \leq k \leq n-2,  \end{cases}
\end{equation}
where $c_{1},\ldots,c_{n-1}$, are constants determined by ${\bf{u}} \in \Delta_{\Lambda}$. Therefore, each $S_{k}$-bundle $E_{k} \to E_{k - 1}$ can be described by means of the equations above and the patterns (inequalities) satisfied by the components of ${\bf{u}} \in \Delta_{\Lambda}$, see for instance \cite[\S 5]{Cho}. From Eq. \ref{CharacpolyLax}
and Lemma \ref{fundequation}, we have that the complex algebraic curve $\mathcal{C}_{Z}$ introduced in Eq. \ref{zerolocuscharacteristic} encodes in its set of branch points some fundamental elements which determine the geometry and topology of the iterated bundle giving in Eq. \ref{itbundle}. In this setting, we have the following theorem:
\begin{theorem}
\label{theo2proof}
Let $(O(\Lambda),\omega_{O(\Lambda)},\mathscr{H})$ be the Gelfand-Tsetlin integrable system associated to some adjoint orbit $O(\Lambda) \subset \mathfrak{u}(n)$ and let $\Delta_{\Lambda}$ be the corresponding Gelfand-Tsetlin polytope. Then, there exists a family of complex algebraic curves
\begin{equation}
\textstyle{\mathcal{M}_{\Lambda} := \Big \{\textstyle{{\mathcal{C}_{Z} = {\rm{Spec}}}\Big(\frac{\mathbb{C}[w,\zeta]}{\langle \zeta^{2} - P_{\Lambda}(w;Z) \rangle}\Big ) }\ \Big | \ Z \in O(\Lambda) \Big\}},
\end{equation}
such that $P_{\Lambda}(w;Z) := r_{\Lambda}(w)\det( w\mathds{1}_{r} - L(Z))$, for every $Z \in O(\Lambda)$, and $r_{\Lambda}(w) \in \mathbb{C}(w)$, satisfying the following:
\begin{enumerate}
\item[1)] For all $Z \in O(\Lambda)$, we have $P_{\Lambda}(w;Z) = \mu_{\Lambda}(w)P_{1}(w;Z) \cdots P_{n-1}(w;Z)$, such that $\mu_{\Lambda}(w)$ is the minimal polynomial of $\Lambda$ and $P_{k}(w;Z) \in \mathbb{C}[w]$, $\forall k = 1,\ldots, n-1$;
\item[2)] Given ${\bf{u}} \in \Delta_{\Lambda}$, then $\mathscr{H}^{-1}({\bf{u}})$ is a Lagrangian torus if and only if $\mathcal{C}_{Z}$ is a smooth hyperelliptic curve, for some $Z \in \mathscr{H}^{-1}({\bf{u}})$;
\item[3)] If ${\bf{u}} \in {\rm{int}}(\mathcal{F})$, for some $d$-dimensional face $\mathcal{F}$ of $\Delta_{\Lambda}$, then $\mathcal{C}_{Z}$ is singular for all $Z \in \mathscr{H}^{-1}({\bf{u}})$;
\item[4)] In the setting of item 3), if $0 \leq {\rm{Res}}\Big(\frac{P'_{k}(w;Z)}{P_{k}(w;Z)};\alpha\Big) \leq 1$, for all $(\alpha,0) \in {\rm{Sing}}(\mathcal{C}_{Z})$, and for all $1 \leq k \leq n-1$, then $\mathscr{H}^{-1}({\bf{u}}) \cong T^{d}$; 
\item[5)] If ${\bf{u}} \in {\rm{int}}(\mathcal{F})$, for some $d$-dimensional face $\mathcal{F}$ of $\Delta_{\Lambda}$, and $\mathscr{H}^{-1}({\bf{u}})$ is a non-torus fiber, then there exists some $(\alpha,0) \in {\rm{Sing}}(\mathcal{C}_{Z})$, for some $Z \in \mathscr{H}^{-1}({\bf{u}})$, such that ${\rm{mult}}_{\alpha}(P_{k}(w;Z)) > 1$, for some $1 \leq k \leq n-1$;
\item[6)] For all $Z \in O(\Lambda)^{\mathscr{H}}$ there exists an Abelian differential $\Omega(Z) \in \mathscr{M}^{1}(\mathcal{C}_{Z})$, with poles $p_{jk}(Z) \in \mathcal{C}_{Z}$, $(j,k) \in \mathcal{I}(\Lambda)$, such that 
\begin{equation}
\label{abeliandescription}
\mathscr{H}(Z) = \Bigg (\frac{1}{2\pi\sqrt{-1}}\oint_{c_{jk}}\Omega(Z)\Bigg)_{(j,k) \in \mathcal{I}(\Lambda)},
\end{equation}
where $c_{jk}$ is a small loop around the point $p_{jk}(Z) \in \mathcal{C}_{Z}$, for every $(j,k) \in \mathcal{I}(\Lambda)$.
\end{enumerate}
\end{theorem}

\begin{proof}
For every $Z \in O(\Lambda)$, consider $\mathcal{C}_{Z} :  \zeta^{2} - P_{\Lambda}(w;Z) = 0$, such that $P_{\Lambda}(w;Z) = r_{\Lambda}(w)\det( w\mathds{1}_{r} - L(Z))$, and define $\mathcal{M}_{\Lambda} := \{\mathcal{C}_{Z} \ | \ Z \in O(\Lambda)\}$. In order to prove item 1), we observe that  
\begin{equation}
P_{\Lambda}(w;Z) = \mu_{\Lambda}(w)F_{\Lambda}(w;Z),
\end{equation}
see Lemma \ref{divisor} and Lemma \ref{fundequation}. Since $\det( w\mathds{1}_{r} - L(Z)) =  \prod_{k = 1}^{n-1}\det\big (w \mathds{1}_{k} - L_{k}(Z)\big)$, we can set
\begin{equation}
P_{k}(w;Z) = {\rm{gcd}}\big (F_{\Lambda}(w;Z),\det( w\mathds{1}_{k} - L_{k}(Z)) \big ), \ \ i = 1, \dots, n-1.
\end{equation}
From this we obtain the decomposition $P_{\Lambda}(w;Z) = \mu_{\Lambda}(w)P_{1}(w;Z) \cdots P_{n-1}(w;Z)$. For item 2), we observe the following. Given ${\bf{u}} \in \Delta_{\Lambda}$, since $\mathscr{H}(Z) = {\bf{u}}$, $\forall Z \in \mathscr{H}^{-1}({\bf{u}})$, we have $\mathcal{C}_{Z} = \mathcal{C}_{Z'}$, $\forall Z,Z'\in \mathscr{H}^{-1}({\bf{u}})$. Indeed, denoting ${\bf{u}} = ({\bf{u}}_{j}^{(k)})_{(j,k) \in \mathcal{I}(\Lambda)}$, it follows from Eq. \ref{minimalspectral} that 
\begin{equation}
\textstyle{\mathcal{C}_{Z} = \big \{ (w,\zeta) \in \mathbb{C}^{2} \ \big | \ \zeta^{2} - \mu_{\Lambda}(w)\big [\prod_{(j,k) \in \mathcal{I}(\lambda)}(w - {\bf{u}}_{j}^{(k)})\big] = 0\big\}}, 
\end{equation}
for all $Z \in \mathscr{H}^{-1}({\bf{u}})$. The description above shows that $\mathcal{C}_{Z}$ is smooth, for some $Z \in  \mathscr{H}^{-1}({\bf{u}})$, if and only if the components of ${\bf{u}}$ are pairwise distinct and ${\rm{gcd}}\big (F_{\Lambda}(w;Z),\mu_{\Lambda}(w)\big) = 1$, i.e., if and only if ${\bf{u}} \in {\rm{int}}(\Delta_{\Lambda})$. Thus, from Corollary \ref{lagranguantorusfiber}, we obtain item 2). The proof of item 3) follows from the following fact. Given a $d$-dimensional face $\mathcal{F}$ of $\Delta_{\Lambda}$, if ${\bf{u}} \in {\rm{int}}(\mathcal{F})$, from the Gelfand-Tsetlin patterns (Eq. \ref{GTpattern}) we have that at least one of the following equalities is satisfied by the components of ${\bf{u}}$:
\begin{center}
${\bf{u}}_{j}^{(k)} = \lambda, \ \ {\bf{u}}_{j}^{(k)} = {\bf{u}}_{j}^{(k+1)}$ \ \ or \ \ ${\bf{u}}_{j}^{(k)} = {\bf{u}}_{j+1}^{(k+1)}$,
\end{center}
for some eigenvalue $\lambda \in \mathbb{C}$ of $\Lambda$. Therefore, if ${\bf{u}} \in {\rm{int}}(\mathcal{F})$, for some $d$-dimensional face $\mathcal{F}$ of $\Delta_{\Lambda}$, we have have that $P_{\Lambda}(w;Z)$ has at least one root with multiplicity $\geq 2$, $\forall Z \in \mathscr{H}^{-1}({\bf{u}})$, which implies that $\mathcal{C}_{Z}$ is singular for all $Z \in \mathscr{H}^{-1}({\bf{u}})$, so we have item 3). Now, given ${\bf{u}} \in {\rm{int}}(\mathcal{F})$, for some $d$-dimensional face $\mathcal{F}$, if we have as stated in item 4) that
\begin{equation}
\textstyle{0 \leq {\rm{Res}}\Big(\frac{P'_{k}(w;Z)}{P_{k}(w;Z)}; \alpha\Big) \leq 1}, 
\end{equation}
for all $(\alpha,0) \in {\rm{Sing}}(\mathcal{C}_{Z})$ and for all $1 \leq k \leq n-1$, it follows that 
\begin{center}
$P_{k}(\alpha;Z) = 0$ \ \ \ if and only if \ \ \ \ ${\rm{mult}}_{\alpha}(P_{k}(w;Z)) = 1,$
\end{center}
for all $(\alpha,0) \in {\rm{Sing}}(\mathcal{C}_{Z})$ and for all $1 \leq k \leq n-1$. On the other hand, since $\mathscr{H}^{-1}({\bf{u}}) \cong T^{d} \times Y({\bf{u}})$ (Theorem \ref{fiberdescription}), if $\dim_{\mathbb{R}}(Y({\bf{u}})) > 0$, it follows that at some stage $1 \leq k_{0} \leq n-1$ of the iterated bundle given in item (i) of Remark \ref{Topfibers} we have a $S^{2\ell-1}$-factor in $S_{k_{0}}$ $(\ell > 1)$. This factor corresponds to the pattern
\begin{center}
$\begin{cases} \lambda_{j} > {\bf{u}}_{j}^{(n-1)} = \lambda_{j+1} = \cdots = \lambda_{j + \ell} = {\bf{u}}_{j + \ell}^{(n-1)} > \lambda_{j + \ell + 1}, \ \ \ \text{if} \ \ k_{0} = n-1,\\ 
\ \ \ \ \ \ \ \ \ \ \ \ \ \ \ \ \ \ \ \ \ \ \ \ \ \ \ \ \ \ \ \ \ \ \ \ \ \ \ \ \ \ \ \ {\text{or}}\\ 
{\bf{u}}_{j}^{(k_{0}+1)} > {\bf{u}}_{j}^{(k_{0})} = {\bf{u}}_{j+1}^{(k_{0}+1)} = \cdots = {\bf{u}}_{j + \ell}^{(k_{0}+1)} = {\bf{u}}_{j + \ell}^{(k_{0})} > {\bf{u}}_{j + \ell + 1}^{(k_{0}+1)}, \ \ \ \text{if} \ \ 1 \leq k_{0} \leq n-2, \end{cases}$
\end{center}
for some $j$, see for instance \cite[Lemma 5.14]{Cho}. Thus, denoting $\alpha =  {\bf{u}}_{j}^{(k_{0})} = \cdots = {\bf{u}}_{j + \ell}^{(k_{0})}$, from above we have $(\alpha,0) \in {\rm{Sing}}(\mathcal{C}_{Z})$, such that ${\rm{mult}}_{\alpha}(P_{k_{0}}(w;Z)) > 1$. Therefore, if $0 \leq {\rm{Res}}\Big(\frac{P'_{k}(w;Z)}{P_{k}(w;Z)};\alpha\Big) \leq 1$, for all $(\alpha,0) \in {\rm{Sing}}(\mathcal{C}_{Z})$ and for all $1 \leq k \leq n-1$, we must have that $\dim_{\mathbb{R}}(Y({\bf{u}})) = 0$, i.e., $\mathscr{H}^{-1}({\bf{u}}) \cong T^{d}$, which concludes the proof of item 4). The proof of item 5) is a consequence of the previous argument, that is, if ${\bf{u}} \in {\rm{int}}(\mathcal{F})$, for some $d$-dimensional face $\mathcal{F}$ of $\Delta_{\Lambda}$, and $\mathscr{H}^{-1}({\bf{u}}) \cong T^{d} \times Y({\bf{u}})$, with $\dim_{\mathbb{R}}(Y({\bf{u}})) > 0$, then there exists $(\alpha,0) \in {\rm{Sing}}(\mathcal{C}_{Z})$, for some $Z \in \mathscr{H}^{-1}({\bf{u}}) $, such that ${\rm{mult}}_{\alpha}(P_{k_{0}}(w;Z)) > 1$, for some $1 \leq k_{0} \leq n-1$.

In order to prove item 6) we observe that, since $O(\Lambda)^{\mathscr{H}} = \mathscr{H}^{-1}({\rm{int}}(\Delta_{\Lambda}))$, for every $Z \in O(\Lambda)^{\mathscr{H}}$ we have that the roots of $P_{\Lambda}(w;Z)$ are pairwise distinct, which implies that the roots of $F_{\Lambda}(w;Z)$ are also pairwise distinct, notice that the roots of $F_{\Lambda}(w;Z)$ are the components of $\mathscr{H}(Z) = (\lambda_{j}^{(k)}(Z))_{(j,k) \in \mathcal{I}(\Lambda)}$. Since in this case $\mathcal{C}_{Z}$ is a hyperelliptic curve, we can take $\Omega(Z) \in \mathscr{M}^{1}(\mathcal{C}_{Z})$, such that
\begin{equation}
\label{AbelGT}
\Omega(Z) := \frac{1}{2}\frac{w F_{\Lambda}'(w;Z)}{F_{\Lambda}(w;Z)}dw =  \sum_{(j,k) \in \mathcal{I}(\Lambda)}\frac{1}{2}\frac{wdw}{w - \lambda_{j}^{(k)}(Z)}.
\end{equation}
The poles of $\Omega(Z)$ on $\mathcal{C}_{Z}$ are given by the branch points $p_{jk}(Z) = ( \lambda_{j}^{(k)}(Z),0) \in \mathcal{C}_{Z}$, where $(j,k) \in \mathcal{I}(\Lambda)$. From a straightforward computation as in Example \ref{prescribedpoles}, it follows that 
$\lambda_{j}^{(k)}(Z) = \frac{1}{2\pi\sqrt{-1}}\oint_{c_{jk}}\Omega(Z)$, for some small loop $c_{jk}$ around the point $p_{jk}(Z) \in \mathcal{C}_{Z}$, for all $(j,k) \in \mathcal{I}(\Lambda)$. From above we obtain item 6) and conclude the proof.
\end{proof}

\subsubsection{Examples and final comments}

In this subsection, we discuss by means of some concrete examples how the results of Theorem \ref{theo2proof} can be applied in the study of the Liouville foliation defined by the Gelfand-Tsetlin integrable systems. The main purpose is to provide a classification for the Gelfand-Tsetlin fibers in terms of elementary tools of the theory of algebraic curves and investigate their relationship with vanishing cycles. 

\begin{example}[Warm-up example]
Consider the Hamiltonian ${\rm{U}}(2)$-space $(O(\Lambda),\omega_{O(\Lambda)},{\rm{U}}(2),\Phi)$, such that $\Lambda = {\text{diag}} \big \{\lambda_{1}, \lambda_{2}\big \}$, where $\lambda_{1} > \lambda_{2}$. In this case we have $O(\Lambda) \cong \mathbb{C}{\rm{P}}^{1}$ and the Gelfand-Tsetlin system $(\mathbb{C}{\rm{P}}^{1},\omega_{\mathbb{C}{\rm{P}}^{1}},\mathscr{H} )$ can be obtained as in Theorem \ref{C5S5.2Teo5.2.10} through the spectral equation 
\begin{equation}
\det (w \mathds{1}_{1} -  L) = (w -  \lambda_{1}^{(1)}) = 0,
\end{equation}
observe that  $L = L_{1}$. Thus, we have $\mathscr{H}(Z) = \lambda_{1}^{(1)}(Z)$, for all $Z \in O(\Lambda)$. By applying Theorem \ref{theo2proof}, we obtain a family of algebraic curves $\mathscr{M}_{\Lambda} = \{\mathcal{C}_{Z} \ | \ Z \in \mathbb{C}{\rm{P}}^{1}\}$, such that
\begin{equation}
\mathcal{C}_{Z} : \zeta^{2} - \mu_{\Lambda}(w)F_{\Lambda}(w;Z) = 0,
\end{equation}
for all $Z \in \mathbb{C}{\rm{P}}^{1}$, where $\mu_{\Lambda}(w) = (w-\lambda_{1})(w-\lambda_{2})$ and $F_{\Lambda}(w;Z) = (w-\lambda_{1}^{(1)}(Z))$. In this case, the Gelfand-Tsetlin polytope $\Delta_{\Lambda}$ is defined by the pattern 
\begin{equation}
\lambda_{2} \leq {\bf{u}}_{1}^{(1)} \leq \lambda_{1},
\end{equation}
and for every ${\bf{u}} = {\bf{u}}_{1}^{(1)} \in \Delta_{\Lambda}$, we have $\mathscr{H}^{-1}({\bf{u}}) = T^{1}$, if ${\bf{u}} \in {\rm{Int}}(\Delta_{\Lambda})$, and $\mathscr{H}^{-1}({\bf{u}}) = {\text{point}}$, if $ {\bf{u}} \in \partial \Delta_{\Lambda}$. Given ${\bf{u}} \in {\rm{Int}}(\Delta_{\Lambda})$, for all $Z \in \mathscr{H}^{-1}({\bf{u}})$ we have that $\mathcal{C}_{Z} \in \mathscr{M}_{\Lambda}$ is a smooth algebraic curve defined by the affine part of a genus 1 Riemann surface $\Sigma_{1}$ (Fig. \ref{curveU2}).

\begin{figure}[H]
\centering\includegraphics[scale=.25]{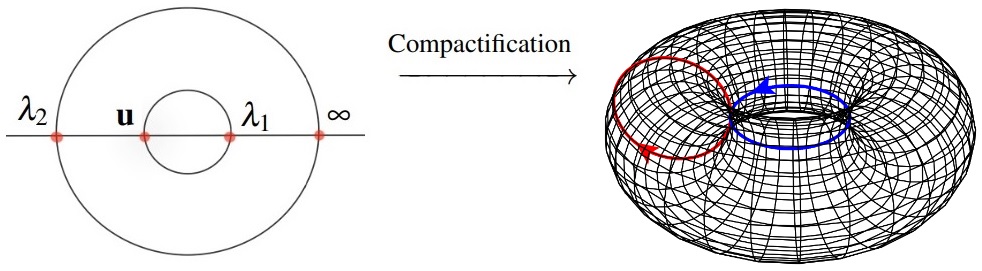}
\caption{On the right-hand side we have the Riemann surface $\Sigma_{1}$ which compactifies $\mathcal{C}_{Z}$.}
\label{curveU2}
\end{figure}

If ${\bf{u}} \in \partial \Delta_{\Lambda}$, we have two possibilities, i.e., ${\bf{u}} = \lambda_{2}$ or ${\bf{u}} = \lambda_{1}$. In both cases, for every $Z \in \mathscr{H}^{-1}({\bf{u}})$, we have that $\mathcal{C}_{Z} \in \mathscr{M}_{\Lambda}$ is a singular algebraic curve (Fig. \ref{fig_GT_U2_1} and Fig. \ref{fig_GT_U2_2}). 
\begin{figure}[H]
\begin{subfigure}{0.5\textwidth}
\centering\includegraphics[scale=.25]{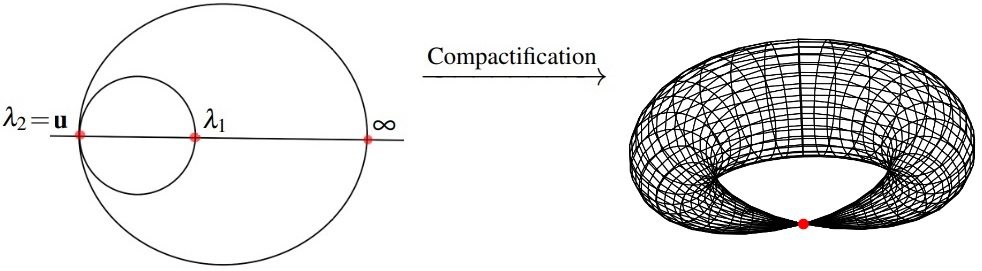} 
\caption{$\mathcal{C}_{Z} : \zeta^{2} - (w-\lambda_{1})(w-\lambda_{2})^{2} = 0$.}
\label{fig_GT_U2_1}
\end{subfigure}
\begin{subfigure}{0.5\textwidth}
\centering\includegraphics[scale=.25]{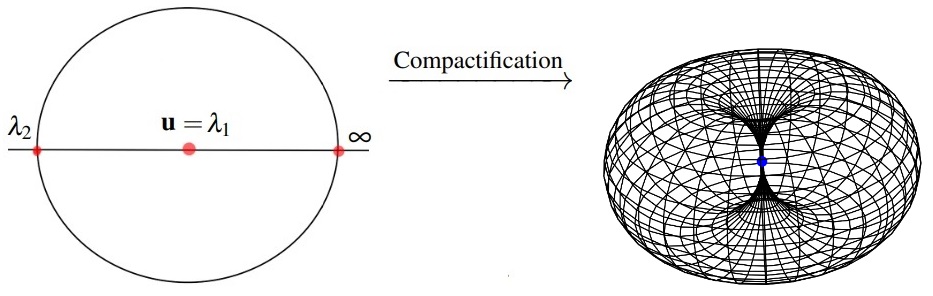}
\caption{$\mathcal{C}_{Z} : \zeta^{2} - (w-\lambda_{1})^{2}(w-\lambda_{2}) = 0$.}
\label{fig_GT_U2_2}
\end{subfigure}
\caption{On the left-hand side we have a pinched torus associated to the singular fiber $\mathscr{H}^{-1}(\lambda_{2})$. On the right-hand side we have a horn torus associated to the singular fiber $\mathscr{H}^{-1}(\lambda_{1})$.}
\label{fig:image2}
\end{figure}
As it can be seen above, the singular fibers of the Gelfand-Tsetlin system correspond to singular curves obtained from the vanish of certain homology cycles of a generic smooth element in $\mathscr{M}_{\Lambda}$ (Fig. \ref{vcyclesrepresentation}). 
\begin{figure}[H]
\centering\includegraphics[scale=.25]{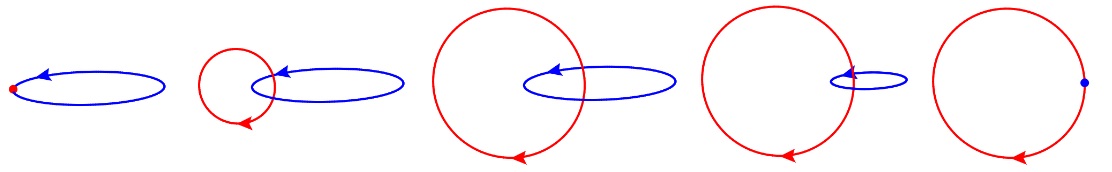}
\caption{Vanishing cycles which characterize the singular fibers of the GT-system.}
\label{vcyclesrepresentation}
\end{figure}

For every ${\bf{u}} \in \Delta_{\Lambda} = [\lambda_{2},\lambda_{1}]$, and every $Z \in \mathscr{H}^{-1}({\bf{u}})$, we have ${\rm{Ord}}_{p}(\mathcal{C}_{Z}) = 1$ or $2$, $\forall p \in \mathcal{C}_{Z}$. Moreover, it is straightforward to show that:
\begin{enumerate}
\item ${\rm{Ord}}_{p}(\mathcal{C}_{Z}) = 1$, $\forall p \in \mathcal{C}_{Z}$, if and only if $\mathscr{H}^{-1}({\bf{u}}) \cong T^{1}$;
\item ${\rm{Ord}}_{p}(\mathcal{C}_{Z}) = 2$, for some $p \in \mathcal{C}_{Z}$, if and only if $\mathscr{H}^{-1}({\bf{u}}) = {\text{point}}$.
\end{enumerate}
From above we have a characterization of the Gelfand-Tsetlin fibers purely in terms of the elements in $\mathscr{M}_{\Lambda}$.
\end{example}

\begin{example}[Generic orbits in $\mathfrak{u}(3)$]
\label{example_U(3)}
Consider the Hamiltonian ${\rm{U}}(3)$-space $(O(\Lambda),\omega_{O(\Lambda)},{\rm{U}}(3),\Phi)$, where $\Lambda = {\text{diag}} \big \{\lambda_{1}, \lambda_{2}.\lambda_{3}\big \}$, such that $\lambda_{1} > \lambda_{2} > \lambda_{3}$. It is straightforward to see that $O(\Lambda) \cong {\rm{Fl}}(3)$, and from Theorem \ref{C5S5.2Teo5.2.10} we have the Gelfand-Tsetlin system $({\rm{Fl}}(3),\omega_{{\rm{Fl}}(3)},\mathscr{H} )$ given by the solutions of the spectral equation 
\begin{equation}
\det (w \mathds{1}_{3} -  L) = (w - \lambda_{1}^{(1)})(w - \lambda_{1}^{(2)})(w - \lambda_{2}^{(2)}) = 0,
\end{equation}
i.e., $\mathscr{H} = (\lambda_{1}^{(2)},\lambda_{2}^{(2)},\lambda_{1}^{(1)})$. Applying Theorem \ref{theo2proof} we get $\mathscr{M}_{\Lambda} = \{\mathcal{C}_{Z} \ | \ Z \in {\rm{Fl}}(3)\}$, such that
\begin{equation}
\mathcal{C}_{Z} : \zeta^{2} - \mu_{\Lambda}(w)F_{\Lambda}(w;Z)= 0,
\end{equation}
for every $Z \in {\rm{Fl}}(3)$, where $\mu_{\Lambda}(w) = (w - \lambda_{1})(w - \lambda_{2})(w - \lambda_{3})$ and $F_{\Lambda}(w;Z) = \det (w \mathds{1}_{3} -  L(Z))$. It is worth mentioning that also from Theorem \ref{theo2proof} we can describe the Guillemin and Sternberg’s action coordinates $\{\lambda_{1}^{(2)},\lambda_{2}^{(2)},\lambda_{1}^{(1)}\}$ on ${\rm{Fl}}(3)^{\mathscr{H}}$ purely in terms of hyperelliptic integrals. In fact, since in this case $F_{\Lambda}(w;-) = \det (w \mathds{1}_{3} -  L)$, by following Eq. \ref{AbelGT}, for every $Z \in {\rm{Fl}}(3)^{\mathscr{H}}$, we have\footnote{Observe that $\partial_{w}\det (w \mathds{1}_{3} -  L(Z)) = \det (w \mathds{1}_{3} -  L(Z)){\rm{tr}}[(w\mathds{1}_{3} - L(Z))^{-1}]$, for every $Z \in {\rm{Fl}}(3)^{\mathscr{H}}$. Thus, after some suitable rearrangement in Eq. \ref{GScoordinateresolvent}, one can also express Guillemin and Sternberg’s action coordinates in terms of the resolvent $\mathscr{R}(w;L(Z)) := (w\mathds{1}_{3} - L(Z))^{-1}$, $w \in \mathbb{C} \backslash \sigma(L(Z))$.}
\begin{equation}
\label{GScoordinateresolvent}
\lambda_{j}^{(k)}(Z) = \frac{1}{4\pi\sqrt{-1}}\oint_{c_{jk}}{\textstyle{{\rm{tr}}\big [(\mathds{1}_{3} -  \frac{1}{w}L(Z))^{-1}\big]dw}}, \ \ \ \ \  \forall 1 \leq j \leq k \leq 2.
\end{equation}
Now, as in the previous example, let us analyze the relationship between singular Gelfand-Tsetlin fibers and vanishing cycles. In this case, the associated Gelfand-Tsetlin polytope is defined by the following patterns: 
\begin{equation}
\label{GTpatternsU(3)}
\begin{matrix} \lambda_{1}&&&&\lambda_{2}&&&&\lambda_{3}\\  
&\rotatebox[origin=c]{-50}{$\geq$}&&\rotatebox[origin=c]{50}{$\geq$}&&\rotatebox[origin=c]{-50}{$\geq$}&&\rotatebox[origin=c]{50}{$\geq$}\\ 
&&{\bf{u}}_{1}^{(2)}&&&&{\bf{u}}_{2}^{(2)}&&\\
&&&\rotatebox[origin=c]{-50}{$\geq$}&&\rotatebox[origin=c]{50}{$\geq$}&&&\\
&&&&{\bf{u}}_{1}^{(1)}&&&&
\end{matrix}
\end{equation}
From above, given ${\bf{u}} \in {\rm{Int}}(\Delta_{\Lambda})$, for all $Z \in \mathscr{H}^{-1}({\bf{u}}) \cong T^{3}$, we have that $\mathcal{C}_{Z} \in \mathscr{M}_{\Lambda}$ is a smooth algebraic curve defined by the affine part of a genus 2 Riemann surface $\Sigma_{2}$ (Fig. \ref{fig_GT_U3} and Fig. \ref{fig_GT_U3_1}).

\begin{figure}[H]
\begin{subfigure}{0.5\textwidth}
\centering\includegraphics[scale=.27]{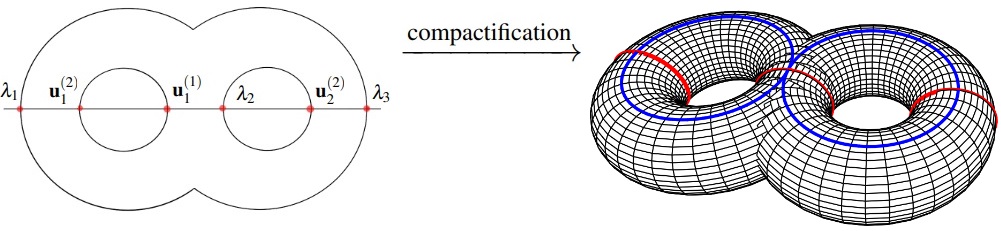} 
\caption{${\bf{u}}_{1}^{(1)} \geq \lambda_{2}$.}
\label{fig_GT_U3}
\end{subfigure}
\begin{subfigure}{0.5\textwidth}
\centering\includegraphics[scale=.27]{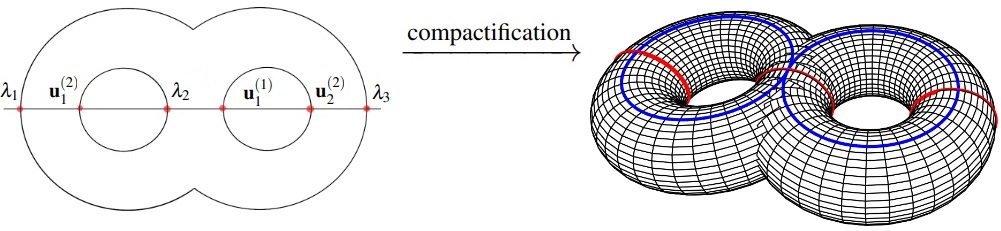}
\caption{$ \lambda_{2}\geq {\bf{u}}_{1}^{(1)}$.}
\label{fig_GT_U3_1}
\end{subfigure}
\caption{Riemann surface which compactifies $\mathcal{C}_{Z}$, for all $Z \in \mathscr{H}^{-1}({\bf{u}})$, such that ${\bf{u}} \in {\rm{Int}}(\Delta_{\Lambda})$.}
\end{figure}
As in the previous example, we can distinguish the singular fibers of the Gelfand-Tsetlin integrable system by means of the vanish of certain homology cycles (Fig. \ref{reghcycles}) of a generic smooth element in $\mathscr{M}_{\Lambda}$.

\begin{figure}[H]
\centering\includegraphics[scale=.30]{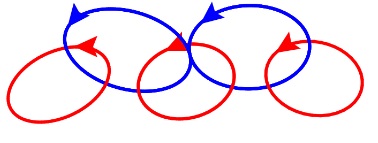}
\caption{The homology cycles for $\Sigma_{2}$.}
\label{reghcycles}
\end{figure}

\begin{figure}[H]
\begin{subfigure}{0.5\textwidth}
\centering\includegraphics[scale=.30]{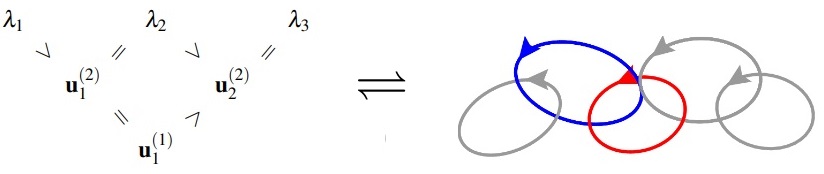} 
\caption{$0$-dimensional face; Associated fiber $\mathscr{H}^{-1}({\bf{u}}) = {\text{point}}$.}
\label{vcycle_1}
\end{subfigure}
\begin{subfigure}{0.5\textwidth}
\centering\includegraphics[scale=.30]{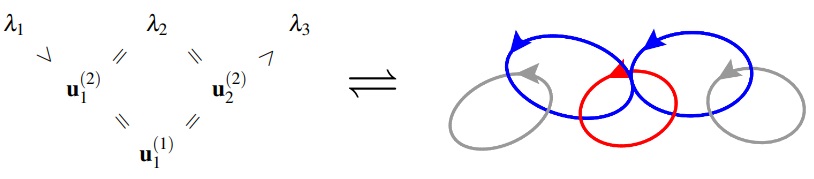}
\caption{$0$-dimensional face; Associated fiber $\mathscr{H}^{-1}({\bf{u}}) \cong S^{3}$.}
\label{vcycle_2}
\end{subfigure}
\begin{subfigure}{0.5\textwidth}
\centering\includegraphics[scale=.30]{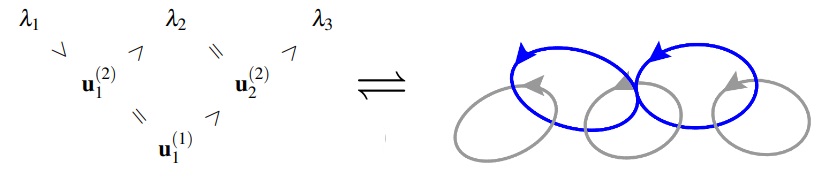}
\caption{$1$-dimensional face; Associated fiber $\mathscr{H}^{-1}({\bf{u}}) \cong T^{1}$.}
\label{vcycle_3}
\end{subfigure}
\begin{subfigure}{0.5\textwidth}
\centering\includegraphics[scale=.30]{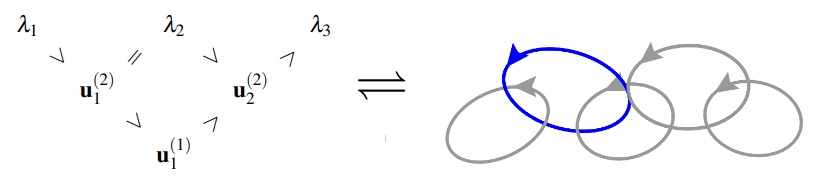}
\caption{$2$-dimensional face; Associated fiber $\mathscr{H}^{-1}({\bf{u}}) \cong T^{2}$.}
\label{vcycle_4}
\end{subfigure}
\caption{Examples of the correspondence between the faces of $\Delta_{\Lambda}$ and vanishing cycles. On the right-hand side of Fig. \ref{vcycle_1}-\ref{vcycle_4} we have the vanishing cycles represented by the blue and red colored circles.}
\label{facesandcycles}
\end{figure}

By looking at the face structure of the underlying GT-polytope $\Delta_{\Lambda}$ described in terms of certain subgraphs of the ladder diagram corresponding to $\Lambda$ (e.g. \cite{AnChoKim}), and observing that $\mathscr{H}^{-1}({\bf{u}}_{1}) \cong \mathscr{H}^{-1}({\bf{u}}_{2})$ if ${\bf{u}}_{1}$ and ${\bf{u}}_{2}$ are contained in the relative interior of the same face of $\Delta_{\Lambda}$, e.g. \cite{Cho}, one has the precise description of the vanishing cycles which defines the singular curve corresponding to each singular fiber of  the Gelfand-Tsetlin system (Fig. \ref{vcycle_1}-\ref{vcycle_4}). From the definition of $\mathscr{M}_{\Lambda}$, the unique\footnote{See for instance \cite[Theorem 1.1]{NU}.} non-torus Lagrangian fiber $\mathscr{H}^{-1}({\bf{u}}) \cong S^{3}$ (Fig. \ref{vcycle_2}) corresponds to the singular algebraic curve 
\begin{equation}
\mathcal{C}_{Z} : \zeta^{2} - (w - \lambda_{1})(w - \lambda_{2})^{4}(w - \lambda_{3}) = 0, \ \ \ \forall Z \in \mathscr{H}^{-1}({\bf{u}}).
\end{equation}
It is straightforward to see from the associated GT-patterns (Eq. \ref{GTpatternsU(3)}) that, $\forall {\bf{u}} \in \partial \Delta_{\Lambda}$, $\mathscr{H}^{-1}({\bf{u}})$ is Lagrangian if and only if there exists $p \in \mathcal{C}_{Z}$, for some $Z \in \mathscr{H}^{-1}({\bf{u}})$, such that ${\rm{Ord}}_{p}(\mathcal{C}_{Z}) = 4$. Joining this last fact with the results of Theorem \ref{theo2proof} we conclude that, for every ${\bf{u}} \in \Delta_{\Lambda}$ and every $Z \in \mathscr{H}^{-1}({\bf{u}})$, we have $1 \leq {\rm{Ord}}_{p}(\mathcal{C}_{Z}) \leq 4$, \ \ \ \ $\forall p \in \mathcal{C}_{Z}$. Moreover, we obtain the following characterization:
\begin{enumerate}
\item (Lagrangian torus) ${\rm{Ord}}_{p}(\mathcal{C}_{Z}) = 1$, $\forall p \in \mathcal{C}_{Z}$, if and only if $\mathscr{H}^{-1}({\bf{u}}) \cong T^{3}$;
\item (Isotropic non-Lagrangian) $1 < {\rm{Ord}}_{p}(\mathcal{C}_{Z}) < 4$, for some $p \in \mathcal{C}_{Z}$, if and only if $\mathscr{H}^{-1}({\bf{u}}) = {\text{point}}$ or $\mathscr{H}^{-1}({\bf{u}}) \cong T^{d}$, $d = 1,2$;
\item (Non-torus Lagrangian) ${\rm{Ord}}_{p}(\mathcal{C}_{Z})=4$, for some $p \in \mathcal{C}_{Z}$, if and only if $\mathscr{H}^{-1}({\bf{u}}) \cong S^{3}$.
\end{enumerate}
Therefore, we have a classification for the leaves of the underlying Liouville foliation in terms of $\mathscr{M}_{\Lambda}$.
\end{example}

\begin{example}[Non-regular orbits]
\label{non-regular}
Consider now the Hamiltonian ${\rm{U}}(4)$-space $(O(\Lambda),\omega_{O(\Lambda)},{\rm{U}}(4),\Phi)$, such that 
\begin{center}
$\Lambda = {\text{diag}} \big \{ \lambda_{1}, \lambda_{2}, \lambda_{3}, \lambda_{4} \big \}$,
\end{center}
where $\lambda_{1} = \lambda_{2} > \lambda_{3} = \lambda_{4}$. In this particular case we have $O(\Lambda) \cong {\rm{Gr}}(2,4)$, and from Theorem \ref{C5S5.2Teo5.2.10} we obtain the Gelfand-Tsetlin integrable system $({\rm{Gr}}(2,4),\omega_{{\rm{Gr}}(2,4)},\mathscr{H} )$ by means of the solutions of the spectral equation
\begin{equation}
\det \big ( w\mathds{1}_{6} - L \big) = \big (w-\lambda_{1}\big)\big (w - \lambda_{2}^{(3)}\big)\big (w-\lambda_{3}\big)\big (w - \lambda_{1}^{(2)}\big)\big (w-\lambda_{2}^{(2)}\big)\big (w-\lambda_{1}^{(1)}\big) = 0.
\end{equation}
From Theorem \ref{theo2proof} we get $\mathscr{M}_{\Lambda} = \{\mathcal{C}_{Z} \ | \ Z \in {\rm{Gr}}(2,4)\}$, such that
\begin{equation}
\mathcal{C}_{Z} : \zeta^{2} - \mu_{\Lambda}(w)F_{\Lambda}(w;Z)= 0,
\end{equation}
for every $Z \in {\rm{Gr}}(2,4)$, where $\mu_{\Lambda}(w) = (w - \lambda_{1})(w - \lambda_{3})$ and
\begin{center}
$F_{\Lambda}(w;Z) = \big (w-\lambda_{1}^{(1)}(Z)\big)\big (w - \lambda_{1}^{(2)}(Z)\big)\big (w-\lambda_{2}^{(2)}(Z)\big)\big (w - \lambda_{2}^{(3)}(Z)\big).$
\end{center}
The associated Gelfand-Tsetlin patterns in this case are given by: 
\begin{equation}
\label{GTpatternGrassmann}
\begin{array}{*{20}c} 
\lambda_{1}&&&&\lambda_{1}&&&&\lambda_{3}&&&&\lambda_{3}\\  
&\rotatebox[origin=c]{-50}{$=$}&&\rotatebox[origin=c]{50}{$=$}&&\rotatebox[origin=c]{-50}{$\geq$}&&\rotatebox[origin=c]{50}{$\geq$}&&\rotatebox[origin=c]{-50}{$=$}&&\rotatebox[origin=c]{50}{$=$}\\ 
&&\lambda_{1}&&&&{\bf{u}}_{2}^{(3)}&&&&\lambda_{3}&&\\
&&&\rotatebox[origin=c]{-50}{$\geq$}&&\rotatebox[origin=c]{50}{$\geq$}&&\rotatebox[origin=c]{-50}{$\geq$}&&\rotatebox[origin=c]{50}{$\geq$}&&\\
&&&&{\bf{u}}_{1}^{(2)}&&&&{\bf{u}}_{2}^{(2)}&&&&\\
&&&&&\rotatebox[origin=c]{-50}{$\geq$}&&\rotatebox[origin=c]{50}{$\geq$}&&&&&\\
&&&&&&{\bf{u}}_{1}^{(1)}&&&&&&
\end{array}
\end{equation}
Given ${\bf{u}} \in {\rm{Int}}(\Delta_{\Lambda})$, as in the previous example, for all $Z \in \mathscr{H}^{-1}({\bf{u}}) \cong T^{4}$, we have that $\mathcal{C}_{Z} \in \mathscr{M}_{\Lambda}$ is a smooth algebraic curve defined by the affine part of a genus 2 Riemann surface $\Sigma_{2}$ (Fig. \ref{fig_GT_U4} and Fig. \ref{fig_GT_U4_1}).
\begin{figure}[H]
\begin{subfigure}{0.5\textwidth}
\centering\includegraphics[scale=.27]{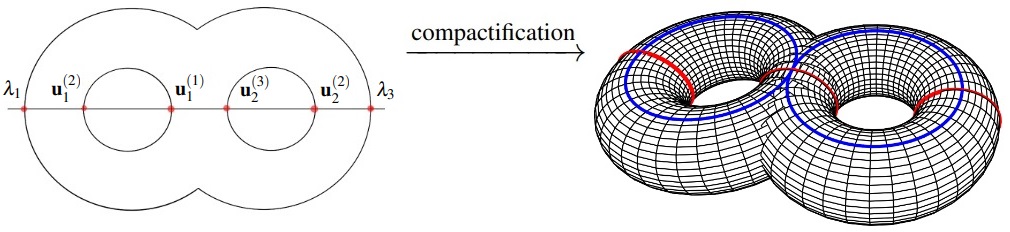} 
\caption{${\bf{u}}_{1}^{(1)} \geq {\bf{u}}_{2}^{(3)}$.}
\label{fig_GT_U4}
\end{subfigure}
\begin{subfigure}{0.5\textwidth}
\centering\includegraphics[scale=.27]{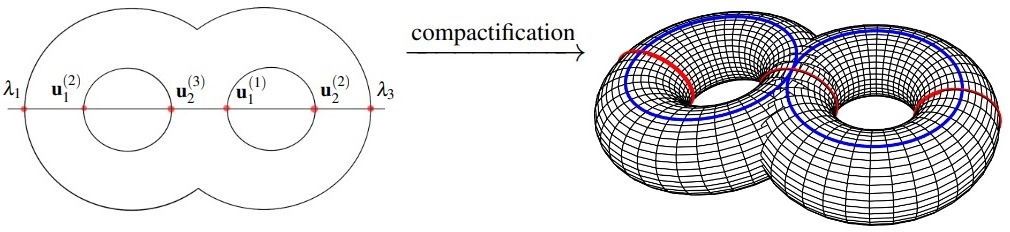}
\caption{$ {\bf{u}}_{2}^{(3)}\geq {\bf{u}}_{1}^{(1)}$.}
\label{fig_GT_U4_1}
\end{subfigure}
\caption{Riemann surface which compactifies $\mathcal{C}_{Z}$, for all $Z \in O(\Lambda)^{\mathscr{H}}$.}
\end{figure}
Under the projection $\mathbb{R}^{4} \to \mathbb{R}$, $\mathscr{H} = \big ( \lambda_{2}^{(3)},\lambda_{2}^{(2)},\lambda_{1}^{(2)},\lambda_{1}^{(1)}\big) \mapsto \lambda_{2}^{(3)}$, we have that $\Delta_{\Lambda}$ is fibered over $[\lambda_{3},\lambda_{1}]$ by Gelfand-Tsetlin polytopes defined by patterns as in Eq. \ref{GTpatternsU(3)}, and the fiber shrinks to a two dimensional triangle on the boundaries $\lambda_{2}^{(3)} = \lambda_{1},\lambda_{3}$ (Fig. \ref{GTpolytopeGr(2,4)}). By taking ${\bf{u}}_{2}^{(3)} = {\bf{u}}_{2}^{(2)} = {\bf{u}}_{1}^{(2)} = {\bf{u}}_{1}^{(1)} = s$, such that $s \in [\lambda_{3},\lambda_{1}]$, we have $\mathscr{H}^{-1}({\bf{u}}) \cong S^{3} \times S^{1}$, if $\lambda_{1} > s > \lambda_{3}$, and $\mathscr{H}^{-1}({\bf{u}}) = {\text{point}}$, if $s = \lambda_{1}$ or $\lambda_{3}$. More precisely, for every ${\bf{u}} \in \partial \Delta_{\Lambda}$, $\mathscr{H}^{-1}({\bf{u}})$ is a Lagrangian fiber if and only if ${\bf{u}} = (s,s,s,s)$, such that $\lambda_{1} > s > \lambda_{3}$. Fibers over other boundary faces are lower dimensional isotropic torus, see for instance \cite[Theorem 1.2]{NU}.
\begin{figure}[H]
\centering\includegraphics[scale=.15]{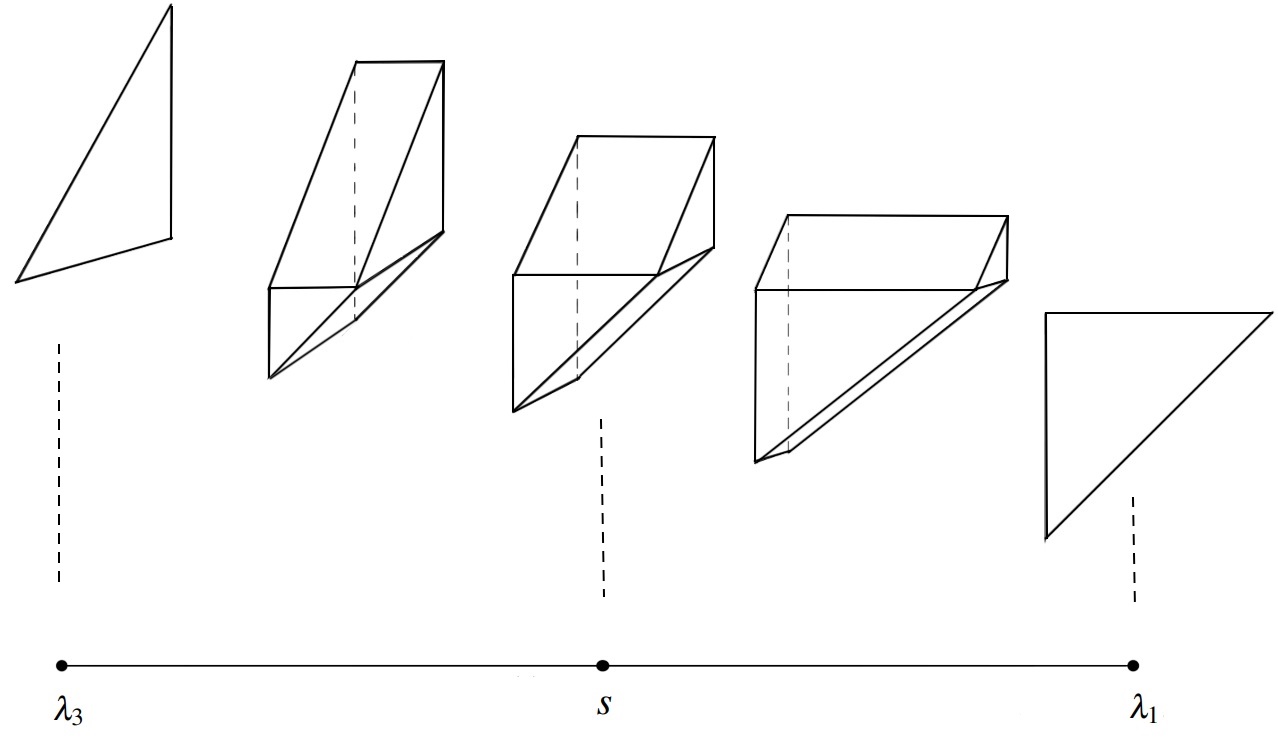}
\caption{The Gelfand-Tsetlin polytope for ${\rm{Gr}}(2,4)$.}
\label{GTpolytopeGr(2,4)}
\end{figure}

From item 1) of Theorem \ref{theo2proof}, for every $Z \in {\rm{Gr}}(2,4)$, we have $F_{\Lambda}(w;Z) = P_{1}(w;Z)P_{2}(w;Z)P_{3}(w;Z)$, such that 
\begin{center}
$P_{1}(w;Z) = \big (w-\lambda_{1}^{(1)}(Z)\big)$, \ \ $P_{2}(w;Z) = \big (w - \lambda_{1}^{(2)}(Z)\big)\big (w-\lambda_{2}^{(2)}(Z)\big)$ \ \ and \ \ $P_{3}(w;Z) = \big (w - \lambda_{2}^{(3)}(Z)\big)$.
\end{center}
Hence, for all $Z \in {\rm{Gr}}(2,4)$, such that ${\rm{Sing}}(\mathcal{C}_{Z}) \neq \emptyset$, we have 
\begin{center}
$0 \leq {\rm{Res}}\Big(\frac{P'_{k}(w;Z)}{P_{k}(w;Z)};\alpha\Big) \leq 1,$
\end{center}
for $k = 1,3$ and for all $(\alpha,0) \in {\rm{Sing}}(\mathcal{C}_{Z})$. In the above setting, note that $1 \leq {\rm{Ord}}_{p}(\mathcal{C}_{Z}) \leq 5$, $\forall p \in \mathcal{C}_{Z}$. Combining these last facts with item 4) and item 5) of Theorem \ref{theo2proof}, we conclude from the Gelfand-Tsetlin patterns (Eq. \ref{GTpatternGrassmann}) that, for every ${\bf{u}} \in \Delta_{\Lambda}$ and every $Z \in \mathscr{H}^{-1}({\bf{u}})$, the following holds:
\begin{enumerate}
\item (Lagrangian torus) ${\rm{Ord}}_{p}(\mathcal{C}_{Z}) = 1$, $\forall p \in \mathcal{C}_{Z}$, if and only if $\mathscr{H}^{-1}({\bf{u}}) \cong T^{4}$;
\item (Isotropic non-Lagrangian) ${\rm{Ord}}_{p}(\mathcal{C}_{Z})=5$, for some $p \in {\rm{Sing}}(\mathcal{C}_{Z})$, or $0 \leq {\rm{Res}}\Big(\frac{P'_{k}(w;Z)}{P_{k}(w;Z)};\alpha\Big) \leq 1$, for all $(\alpha,0) \in {\rm{Sing}}(\mathcal{C}_{Z})$, and for all $1 \leq k \leq 3$, if and only if $\mathscr{H}^{-1}({\bf{u}}) = {\text{point}}$ or $\mathscr{H}^{-1}({\bf{u}}) \cong T^{d}$, $d = 1,2,3$;
\item (Non-torus Lagrangian) ${\rm{Ord}}_{p}(\mathcal{C}_{Z})=4$ and ${\rm{mult}}_{\alpha}(P_{2}(w;Z)) > 1$, for some $p = (\alpha,0) \in {\rm{Sing}}(\mathcal{C}_{Z})$, if and only if $\mathscr{H}^{-1}({\bf{u}}) \cong S^{3} \times S^{1}$.

\end{enumerate}
As in the previous example, we have a classification for the leaves of the Liouville foliation defined by the Gelfand-Tsetlin integrable system $({\rm{Gr}}(2,4),\omega_{{\rm{Gr}}(2,4)},\mathscr{H} )$ in terms of $\mathscr{M}_{\Lambda}$.
\end{example}

%%%%%%%%%%%%%%%%%%%%%%%%%%%%%%%%%%%%%%%%%%%%%%%%%%%%%%%%%%%%%%%%%%%%%%%%%%%%%%%%%%%%%%%%%%%%%%%%%%%%%%%%%%%%%%%%%%%%%%%%%%%%%%%%%%%%%%%%%%%%%%%%%%%%%%%%%%%%%%%%%%%%%%%

%\Addresses


\begin{thebibliography}{BGGSM}

\bibitem{Adler} Adler, M.; van Moerbeke, P.; Completely integrable systems, Euclidean Lie algebras, and curves, Advances in Math. 38, (1980), 267-379.

%\bibitem{Alamiddine} Alamiddine, I.; G\'{e}om\'{e}trie de syst\`{e}mes Hamiltoniens int\'{e}grables: le cas du syst\`{e}me de
%Gelfand–Ceitlin. PhD thesis, Universit\'{e} Toulouse III-Paul Sabatier, 2009.

\bibitem{AnChoKim} An, B. H.; Cho, Y.; Kim, J. S.; On the f-vectors of Gelfand-Cetlin polytopes, Eur. J. Comb. 67 (2018) 61-77.

\bibitem{Arnold2} Arnold, V. I.; Hamiltonian nature of Euler equations of dynamics of rigid body in ideal fluid. Uspekhi Mat. Nauk 26, No.3, 225-226 (1969).

\bibitem{Arnold1} Arnold, V. I.; Sur la geometrie des groupes de Lie de dimension infinie et ses applications a l'hydrodynamique des fluides parfaites. Ann. Inst. Fourier 16, 319-361 (1966).

\bibitem{Audin} Audin, M.; Spinning Tops: A Course on Integrable Systems. Cambridge Studies in Advanced Mathematics. Cambridge University Press, 1999.

\bibitem{Audin1} Audin, M.; Silhol, R.; Vari\'{e}t\'{e}s abéliennes r\'{e}elles et toupie de Kowalevski. Compositio Mathematica 87, no. 2 (1993): 153-229.


\bibitem{INTCLASSICALSYSTEM} Babelon, O.; Bernard, D.; Talon, M.; Introduction to Classical Integrable Systems, Cambridge Monographs on Mathematical Physics (2007).

\bibitem{Beauville} Beauville, A.; Jacobiennes des courbes spectrales et syst\`{e}mes hamiltoniens compl\`{e}tement int\'{e}grables, Acta Math. 164 (1990), 211-235.

\bibitem{Bosch} Bosch, S.; Algebraic Geometry and Commutative Algebra, Universitext, Springer (2013).

\bibitem{BolsinovOshemkov} Bolsinov, A. V.; Oshemkov, A. A.; Singularities of integrable Hamiltonian systems, Topological methods in the theory of integrable systems, pp. 1-67. Cambridge Scientific Publishing, Cambridge (2006).

%\bibitem{besse} Besse, A.; Einstein manifolds. Springer; Berlin Heidelberg New York 1987 edition (2007).

\bibitem{Bouloc} Bouloc, D.; Miranda, E.; Zung, N. T.; Singular fibers of the Gelfand-Cetlin system on $\mathfrak{u}(n)^{\ast}$. Phil. Trans. R. Soc. A. (2018).

%\bibitem{Brown} Brown, L. S.; Quantum Field Theory, Cambridge University Press; Revised edition (August 26, 1994).

\bibitem{QUANTUM} Chari, V.; Pressley, A. N.; A Guide to Quantum Groups, Cambridge University Press; Reprint edition (1995).

\bibitem{Donagi} Donagi, R.; Spectral covers, Math. Sci. Res. Inst. Publ., 28 (1995), 65-86.

\bibitem{FUTORNY2} Drozd, Yu. A.; Futorny, V.; A. Ovsienko, S.; Harish-Chandra suabalgebra and Gelfand-Zetlin modules. Finite dimensional algebras and Related topics, Series, Math. and Phys. Sci., v. 424. p. 79-93 (1992).

\bibitem{Eliasson} Eliasson, L. H.; Normal forms for Hamiltonian systems with Poisson commuting integrals-elliptic case. Commentarii Mathematici Helvetici 65 (1990): 4-35.

\bibitem{Euler} Euler, L.; Decouverte d'une nouveau principe de mechanique. Mere. Acad. Sci. Berlin 14, 154-193 (1758). 

\bibitem{FEDDEV} Faddeev, L. D.; Instructive history of the quantum inverse scattering method. In: Quantum field theory: perspective and prospective (Les Houches, 1998), 161–176, NATO Sci. Ser. C Math. Phys. Sci., 530, Kluwer Acad. Publ., Dordrecht, 1999.

\bibitem{Fischer} Fischer, G.; Plane Algebraic Curves, Student Mathematical Library, vol. 15, American Mathematical Society, Providence (2001).

\bibitem{JPF} Fran\c{c}oise, J. -P.; Calculs explicites d'action-angles, S\'{e}minaire de Math\'{e}matiques Sup\'{e}rieures de Montr\'{e}al, (G. Sabidussi, \'{e}d., collig\'{e} par P. Winternitz) 102 (1986), 101-120.

\bibitem{JPFTarama} Fran\c{c}oise, J. -P.; Tarama, D.; Analytic extension of the Birkhoff normal forms for the free rigid body dynamics on ${\rm{SO}}(3)$. Nonlinearity 28 (2015), no. 5, 1193-1216.

\bibitem{JPFTarama1} Fran\c{c}oise, J. -P.; Tarama, D.; The Rigid Body Dynamics in an Ideal Fluid : Clebsch Top and Kummer Surfaces in ``Integrable Systems and Algebraic Geometry" vol. 2 (Edts R. Donagi and T. Shaska), London Mathematical Society Lecture Note Series 459, 288-312 (2020).

\bibitem{FUTORNY} Futorny, V.; Grantcharov, D.; Ramirez, L. E.; Singular Gelfand-Tsetlin modules of gl(n). Advances in Mathematics (2015).

%\bibitem{Flaschka} Flaschka, H.; The Toda lattice. II. Existence of integrals. Physical Review B, 9(4), 1924-1925 (1974).

\bibitem{Giacobbe} Giacobbe, A.; Some remarks on the Gelfand-Cetlin system. J. Phys. A, 35(49):10591-10605, 2002.

\bibitem{Girondo} Girondo, E.; Gonz\'{a}lez-Diez G.; Introduction to compact Riemann surfaces and dessins d'enfants, London Mathematical Society Student Texts Vol. 79 (Cambridge University Press, 2011).

\bibitem{Griffiths} Griffiths, P. A.; Linearizing flows and a cohomological interpretation of Lax equations. Am. J. Math. 107, 1445-1483 (1985).

\bibitem{Guest} Guest, M. A.; Harmonic Maps, Loop Groups, and Integrable Systems, Cambridge University Press; (1997).

%\bibitem{Symplecticfibration} Guillemin, V.; Lerman, E.; Sternberg, S.; Symplectic Fibrations and Multiplicity Diagrams, Cambridge University Press; 1 edition (2009).

\bibitem{QUANTMULTFREE} Guillemin, V.; Sternberg, S.; Geometric quantization and Multiplicities of Group Representations, Invent. Math. 67, 515-538 (1982).

%\bibitem{MOMENTSREDUC} Guillemin, V.; Sternberg, S.; Moments and Reductions, Differential Geometric Methods in Mathematical Physics: Proceedings of a Conference Held at the Technical University of Clausthal, FRG, 23-25 (1980).

\bibitem{INTMULT} Guillemin, V.; Sternberg, S.; On the collective complete integrability according to the method of Thimm. Ergodic Theory 3, 219-230 (1983).

%\bibitem{SYMPPHY} Guillemin, V.; Sternberg, S.; Symplectic Techniques in Physics, Cambridge University Press (1990).

\bibitem{GELFORB} Guillemin, V.; Sternberg, S.; The Gelfand-Cetlin system and quantization of the complex flag manifolds, J. Funct., Anal 52, 106-128 (1983).

\bibitem{COL} Guillemin, V.; Sternberg, S.; The moment map and collective motion. Ann. Phys. 127 (1980), 220-253.

%\bibitem{MEGUMI} Harada, M.; The symplectic geometry of the Gelfand-Cetlin-Molev basis for representations of ${\rm{Sp}}(2n,\mathbb{C})$. Thesis (Ph. D. in Mathematics), University of California, Berkeley, Spring 2003.

\bibitem{Heckman} Heckman, G. J.; Projections of orbits and asymptotic behavior of multiplicities for compact connected Lie
groups. Invent. Math. 67(2), 333-356 (1982).

\bibitem{Hitchin} Hitchin, N. J.; Stable bundles and integrable systems. Duke Math. J., 54 (1987), 91-114.

\bibitem{Hwang} Hwang, S. G.; Cauchy’s Interlace theorem for Eigenvalues of Hermitian Matrices, The American Mathematical Monthly, Vol. 111, No. 2 (2004), pp. 157-159.


%\bibitem{KAZHDAN} Kazhdan, D.; Kostant, B.; Sternberg, S.; Hamiltonian group actions and dynamical systems of Calogero type, Comm. Pure Appl. Math. 31, 481-508 (1978).

\bibitem{Ito} Ito, H.; Convergence of Birkhoff normal forms for integrable systems. Commentarii Mathematici Helvetici 64 (1989): 412-61.

\bibitem{Izosimov} Izosimov, A., Singularities of Integrable Systems and Algebraic Curves, Int. Math. Res. Notices, 2016, vol.2016, 50pp.

\bibitem{Kirwan} Kirwan, F.; Complex Algebraic Curves, London Mathematical Society Student Texts, Cambridge University Press; 1 edition (1992).

\bibitem{Knapp} Knapp, A. W.; Advanced Algebra, Birkh\"{a}user; 1st ed. 2008.

\bibitem{Lane} Lane, J.; Convexity and Thimm’s trick, Transform. Groups 23 (2018), no. 4, 963–987.

\bibitem{Lax} Lax, P. D.; Integrals of nonlinear equations of evolution and solitary waves, Commun. Pure Appl. Math. 21 (1968) 467-490.

\bibitem{Lerman} Lerman, E.; Meinrenken, E.; Tolman, S.; Woodward, Ch.; Non-abelian convexity by symplectic cuts, Topology 37 (1998), no. 2, 245-259.

\bibitem{INTROMECH} Marsden, J. E.; Ratiu, T.; Introduction to Mechanics and Symmetry: A Basic Exposition of Classical Mechanical Systems, Texts in Applied Mathematics, Springer, 2nd edition (2002).

\bibitem{MirandaZung} Miranda, E.; Zung, N. T.; Equivariant normal form for nondegenerate singular orbits of integrable Hamiltonian systems. Annales Scientifiques de l'Ecole Normale Sup\'{e}rieure 37 (2004): 819-39.

\bibitem{Moser} Moser, J.; Integrable Hamiltonian Systems and Spectral Theory. Accademia Nazionale dei Lincei Scuola Normale Superiore, Pisa, 1981.

%\bibitem{MW} Marsden, J. E.; Weinstein, A.; Reduction of symplectic manifolds with symmetry, Rep. Mathematical Phys. 5 (1974), 121-130.

\bibitem{NarukiTarama} Naruki, I.; Tarama, D.; Some elliptic fibrations arising from free rigid body dynamics, Hokkaido Math. J., 41(3), 365-407, 2012.

\bibitem{NNU} Nishinou, T.; Nohara Y.; Ueda, K.; Toric degenerations of Gelfand-Cetlin systems and potential functions. Adv. Math. 224 (2010), no. 2, 648-706.

\bibitem{NU} Nohara, Y.; Ueda, K.; Floer cohomologies of non-torus fibers of the Gelfand-Cetlin system. J. Symplectic. Geom. 14 (2016),
No. 4, 1251-1293.

\bibitem{Panyushev} Panyushev, D. I.; and Yakimova, O. S.; Poisson-commutative subalgebras and complete integrability on non-regular coadjoint orbits and flag varieties. Math. Z. (2019).

\bibitem{Perelomov} Perelomov, A. M.; Integrable Systems of Classical Mechanics and Lie Algebras, Springer Verlag, 1990.

\bibitem{Rudolph} Rudolph, G. ; Schmidt, M. ; Differential Geometry and Mathematical Physics - Part I: Manifolds, Lie Groups and Hamiltonian Systems, Springer-Verlag (2013).

\bibitem{Sepanski} Sepanski, M. R.; Compact Lie Groups, Graduate Texts in Mathematics, Springer (2007).

\bibitem{Schlag} Schlag, W.; A Course in Complex Analysis and Riemann Surfaces, Amer Mathematical Society, Graduate Studies in Mathematics
Volume: 154; 2014.

\bibitem{SKY} Sklyanin, E. K.;  On complete integrability of the Landau–Lifschitz equation. Preprint LOMI E-3-79, Leningrad, 1979.
Differential Geometry of Singular Spaces and Reduction of Symmetry, New Mathematical Monographs (Book 23), Cambridge University Press; 1 edition (2013).

\bibitem{Singspaces} \'{S}niatycki, J.; Differential Geometry of Singular Spaces and Reduction of Symmetry, New Mathematical Monographs (Book 23), Cambridge University Press; 1 edition (2013).

\bibitem{TaramaJPF} Tarama, D.; Fran\c{c}oise, J. -P.; Analytic extension of Birkhoff normal forms for Hamiltonian systems of one degree of freedom-simple pendulum and free rigid body dynamics. Exponential analysis of differential equations and related topics, 219-236, RIMS K\^{o}ky\^{u}roku Bessatsu, B52, Res. Inst. Math. Sci. (RIMS), Kyoto, 2014.

\bibitem{THIMM} Thimm, A.; Integrable geodesic flows on homogeneous spaces, Ergodic Theory and Dynamical Systems I (1980), 495-5 17.

%\bibitem{QUANTUMG} Vyjayanthi, C.; N. Pressley, A.; A Guide to Quantum Groups, Cambridge University Press (1995).

\bibitem{Vey} Vey, J.; Sur certains systèmes dynamiques séparables. American Journal of Mathematics 100 (1978): 591-614.

\bibitem{WARNER} Warner, F. W.; Foundations of Differentiable Manifolds and Lie Groups, Graduate Texts in Mathematics (Book 94); Springer (1983).

%\bibitem{Woodward} Woodward, Ch.; Multiplicity free Hamiltonian actions need not be K\"{a}hler, Invent. Math. 131 (1998), no. 2, 311-319.

\bibitem{Cho} Yunhyung, C.; Yoosik, K.; Yong-Geun, O.; Lagrangian fibers of Gelfand-Cetlin systems. Advances in Mathematics, Vol. 372, 107304, 2020. 

\bibitem{Zakharov} Zakharov, V. E.; Shabat, A. B.; Integration of nonlinear equations of mathematical physics by the method of inverse scattering, Funct. Anal. Appl. 13 (1979) 166-174.

\end{thebibliography}
\end{document}